%% file: main.tex
\documentclass[10pt, a4paper, twoside]{amsart}

\input{style}
\input{commands}

\graphicspath{{Fig/}}
 
\title[Continuous-Time Accelerated Methods via a Hybrid Control Lens]
{Continuous-Time Accelerated Methods via a Hybrid Control Lens}

\author[A. Sharifi Kolarijani, P. Mohajerin Esfahani, T. Keviczky]{Arman Sharifi Kolarijani, Peyman Mohajerin Esfahani, Tam\'{a}s Keviczky}
	\thanks{The authors are with the Delft Center for Systems and Control, TU Delft, The Netherlands ({\tt \{a.sharifikolarijani,p.mohajerinesfahani,t.keviczky\}@tudelft.nl}).}

\begin{document}
\maketitle

\begin{abstract}
Treating optimization methods as dynamical systems can be traced back centuries ago in order to comprehend the notions and behaviors of optimization methods. Lately, this mind set has become the driving force to design new optimization methods. Inspired by the recent dynamical system viewpoint of Nesterov's fast method, we propose two classes of fast methods, formulated as hybrid control systems, to obtain pre-specified exponential convergence rate. Alternative to the existing fast methods which are parametric-in-time second order differential equations, we dynamically synthesize feedback controls in a state-dependent manner. Namely, in the first class the damping term is viewed as the control input, while in the second class the amplitude with which the gradient of the objective function impacts the dynamics serves as the controller. The objective function requires to satisfy the so-called Polyak--{\L}ojasiewicz inequality which effectively implies no local optima and a certain gradient-domination property. Moreover, we establish that both hybrid structures possess Zeno-free solution trajectories. We finally provide a mechanism to determine the discretization step size to attain an exponential convergence rate.
\end{abstract}

\section{Introduction}
\label{sec:intro}
There is a renewed surge of interest in gradient-based algorithms in many computational communities such as machine learning and data analysis. The following non-exhaustive list of references indicates typical application areas: clustering analysis \cite{lashkari2008convex}, neuro-computing \cite{bottou1991stochastic}, statistical estimation \cite{salakhutdinov2003optimization}, support vector machines \cite{allen2016katyusha}, signal and image processing \cite{becker2011nesta}, and networked-constrained optimization \cite{ghadimi2013multi}. This interest primarily stems from low computational and memory loads of these algorithms (making them exceptionally attractive in large-scale problems where the dimension of decision variables can be enormous). As a result, a deeper understating of how these algorithms function has become a focal point of many studies. 

One research direction that has been recently revitalized is the application of ordinary differential equations (ODEs) to the analysis and design of optimization algorithms. Consider an iterative algorithm that can be viewed as a discrete dynamical system, with the scalar $s$ as its step size. As $s$ decreases, one can observe that the iterative algorithm in fact recovers a differential equation, e.g., in the case of gradient descent method applied to an unconstrained optimization problem $\min_{X\in\mathbb{R}^n}~{\small f(X)}$, one can inspect that   
\begin{equation*}
\begin{array}{c}
X^{k+1}=X^k-s \nabla  f(X^k) ~ \leadsto ~ \dot{X}(t)=-\nabla f\big(X(t)\big)
\end{array}
\end{equation*}
where $f:\mathbb{R}^n\rightarrow \mathbb{R}$ is a smooth function, $X$ is the decision variable, $k\in \mathbb{Z}_{\geq 0}$ is the iteration index, and $t\in \mathbb{R}_{\geq 0}$ is the time. The main motivation behind this line of research has to do with well-established analysis tools in dynamical systems described by differential equations. 

The slow rate of convergence of the gradient descent algorithm ($\mathcal{O}(\frac{1}{t})$ in continuous and $\mathcal{O}(\frac{1}{k})$ in discrete time), limits its application in large-scale problems. In order to address this shortcoming, many researchers resort to the following class of 2nd-order ODEs, which is also the focus of this study:
\begin{equation}
\label{dyn2}
\ddot{X}(t)+\gamma(t)\dot{X}(t)+\nabla f\big(X(t)\big)=0.
\end{equation}
Increasing the order of the system dynamics interestingly helps improve the convergence rate of the corresponding algorithms to $\mathcal{O}(\frac{1}{k^2})$ in the discrete-time domain or to $\mathcal{O}(\frac{1}{t^2})$ in the continuous-time domain. Such methods are called \emph{momentum}, \emph{accelerated}, or \emph{fast} gradient-based iterative algorithms in the literature. The time-dependent function $\gamma:\mathbb{R}_{\geq 0}\rightarrow \mathbb{R}_{>0}$ is a \emph{damping} or a \emph{viscosity} term, which has also been referred to as the \emph{asymptotically vanishing viscosity} since $\lim_{t\rightarrow \infty}~\gamma(t)=0$ \cite{Cabot2004steepest}. 

\textbf{Chronological developments of fast algorithms:} 
It is believed that the application of (\ref{dyn2}) to speed-up optimization algorithms is originated from \cite{Polyak1964} in which Polyak was inspired by a physical point of view (i.e., a heavy-ball moving in a potential field). Later on, Nesterov introduced his celebrated accelerated gradient method in \cite{Nesterov1983} using the notion of ``{estimate sequences}" and guaranteeing convergence rate of $\mathcal{O}(\frac{1}{k^2})$. 
Despite several extensions of Nesterov's method  \cite{Nesterov2004,Nesterov2005smooth,Nesterov2013gradient}, the approach has not yet been fully understood.  In this regard, many have tried to study the intrinsic properties of Nesterov's method such as \cite{Drusvyatskiy2016,Bubeck2015geometric,Drori2014,Lessard2016}. Recently, the authors in \cite{Su2014differential} and in details \cite{Su2016differential} surprisingly discovered that Nesterov's method recovers (\ref{dyn2}) in its continuous limit, with the time-varying damping term $\gamma (t)=\frac{3}{t}$.

\textbf{A dynamical systems perspective:} Based on the observation suggested by \cite{Su2014differential}, several novel fast algorithms have been developed. Inspired by the mirror descent approach \cite{nemirovskii1983problem}, the ODE (\ref{dyn2}) has been extended to non-Euclidean settings using the Bregman divergence in \cite{krichene2015accelerated}. 
Then, the authors in \cite{Wibisono2016variational} further generalized the approach in \cite{krichene2015accelerated} to higher order methods using instead the Bregman Lagrangian. 
Following \cite{Wibisono2016variational}, a ``{rate-matching}" Lyapunov function is proposed in \cite{wilson2016lyapunov} with its monotonicity property established for both continuous and discrete dynamics. Recently, the authors in \cite{Lessard2016} make use of an interesting semidefinite programming framework developed by \cite{Drori2014} and use tools from robust control theory to analyze the convergence rate of optimization algorithms. More specifically, the authors exploit the concept of integral quadratic constraints (IQCs) \cite{megretski1997system} to design iterative algorithms under the strong convexity assumption. Later, the authors in \cite{fazlyab2017analysis} extend the results of IQC-based approaches to quasi-convex functions. The authors in \cite{hu2017dissipativity} use dissipativity theory \cite{willems1972dissipative} along with the IQC-based analysis to construct Lyapunov functions enabling rate analyses. 
In \cite{attouch2016fast}, the ODE~\eqref{dyn2} is amended with an extra Hessian driven damping $\beta \nabla^2 f(X(t))$ for some positive scalar $\beta$. It is shown that the proposed dynamics can be generalized to the case of lower-semicontinuous functions via an appropriate reparameterization of the dynamics. 
The authors in \cite{krichene2016adaptive} propose an averaging approach to construct a broad family of fast mirror descent methods. They also introduce a state-dependent, heuristic method to adaptively update the averaging function.

\textbf{Restarting schemes:} A characteristic feature of fast methods is the non-monotonicity in the suboptimality measure $f-f^*$, where $f^*$ refers to the optimal value of function $f$. The reason behind such an undesirable behavior can be intuitively explained in two ways: (i) a momentum based argument indicating as the algorithm evolves, the algorithm's momentum gradually increases to a level that it causes an oscillatory behavior \cite{o2015adaptive}; (ii) an acceleration-based argument indicating that the asymptotically vanishing damping term becomes so small that the algorithm's behavior drifts from an over-damped regime into an under-damped regime with an oscillatory behavior \cite{Su2016differential}. To prevent such an undesirable behavior in fast methods, an optimal fixed restart interval is determined in terms of the so-called condition number of function $f$ such that the momentum term is restarted to a certain value, see e.g., \cite{Nesterov2004,nemirovski2005efficient,gu2013parnes,lan2013iteration,Nesterov2013gradient}.
It is worth mentioning that \cite{o2015adaptive} proposes two heuristic adaptive restart schemes. It is numerically observed that such restart rules practically improve the convergence behavior of a fast algorithm.

\textbf{Regularity for exponential convergence:} 
Generally speaking, exponential convergence rate and the corresponding regularity requirements of the function $f$ are two crucial metrics in fast methods. In what follows, we discuss about these metrics for three popular fast methods in the literature. 
(Notice that these fast methods are in general designed for wider classes of functions and not limited to the specific cases reported below.) 
When the objective functions are strongly convex with a constant $\sigma_f$ and their gradient is Lipschitz with a constant $L_f$, \cite{Su2016differential} proposes the ``{speed restarting}" scheme
\begin{equation*}
\text{sup}\Big\{ t>0:~\forall \tau\in(0,t),{\small \frac{d\| \dot{X}(\tau)  \|^2}{d\tau}}>0 \Big\},
\end{equation*}
to achieve the convergence rate of:
\begin{equation*}
f\big(X(t)\big)-f^* \leq d_1 e^{-d_2 t} \| X(0) -X^* \|^2.
\end{equation*}
The positive scalars $d_1$ and $d_2$ depend on the constants $\sigma_f$ and $L_f$. 
Assuming the convexity of the function $f$ with a certain choice of parameters in their ``{ideal scaling}" condition, \cite{Wibisono2016variational} uses the dynamics
\begin{align*}
\ddot{X}(t)+c\dot{X}(t)
+c^2e^{ct} \Big( \nabla^2 h\big(X(t)+\frac{1}{c}\dot{X}(t)\big) \Big)^{-1}\nabla f\big(X(t)\big)=0,
\end{align*}
and guarantees the convergence rate of $\mathcal{O}(e^{-ct})$ for some positive scalar $c$, 
where the function $h$ is a distance generating function. 
Under uniform convexity assumption with a constant $\nu_f$, it is further shown that
\begin{equation*}
f\big(X(t)\big)-f^* \leq \Big(f\big(X(0)\big)-f^*\Big) e^{-\nu_f \frac{1}{p-1}t}.
\end{equation*}
where $p-1$ is the order of smoothness of $f$. The authors in \cite{wilson2016lyapunov} introduce the Lyapunov function
\begin{equation*}
\mathcal{E}(t)=e^{\beta(t)}\left( f\big(X(t)\big)-f^*+\frac{\sigma_f}{2} \| X^*-Z(t)  \|^2  \right),
\end{equation*}
to guarantee the rate of convergence
 \begin{equation*}
 \mathcal{E}(t)  \leq \mathcal{E}(0) e ^{-\int \dot{\beta}(s) ds},
 \end{equation*}
where $Z(t)=X(t)+\frac{1}{\dot{\beta}(t)}\dot{X}$, $\dot{Z}(t)=-\dot{X}(t)-\frac{1}{\sigma_f} \dot{\beta}(t)\nabla f\big(X(t) \big)$, and $\beta(t)$ is a user-defined function.

\textbf{Statement of hypothesis:} Much of the references reviewed above (excluding, e.g., \cite{attouch2016fast} and \cite{krichene2016adaptive}) primarily deal with constructing a time-dependent damping term $\gamma(t)$ that
is sometimes tied to a Lyapunov function. Furthermore, due to underlying oscillatory behavior of the corresponding 2nd-order ODE, researchers utilize restarting schemes to over-write the steady-state non-monotonic regime with the transient monotonic regime of the dynamics. In general, notice that these schemes are based on time-dependent schedulers.

With the above argument in mind, let us view an algorithm as a unit point mass moving in a potential field caused by an objective function $f$ under a parametric (or possibly constant) viscosity, similar to the second order ODE~\eqref{dyn2}. In this view, we aim to address the following two questions:
\begin{flushleft}
Is it possible to
\begin{enumerate}[label=(\Roman*)]
\item synthesize the damping term $\gamma$ as a state-dependent term (i.e., $\gamma(X,\dot X)$), or
\item dynamically control the magnitude of the potential force $\nabla f(X)$,
\end{enumerate}
such that the underlying properties of the optimization algorithm are improved?
\end{flushleft}

\textbf{Contribution:} In this paper, we answer these questions by amending the 2nd-order ODE~\eqref{dyn2} in two ways as follows:
\begin{align*}
\text{(I)} &~ \ddot{X}(t)+u_{\textbf{I}}\big(X(t),\dot{X}(t)\big)~\dot{X}(t)+\nabla f(X(t))=0,\\
\text{(II)} &~ \ddot{X}(t)+\dot{X}(t)+u_{\textbf{II}}\big(X(t),\dot{X}(t)\big)~\nabla f(X(t))=0,
\end{align*}
where the indices indicate to which question each structure is related to in the above hypothesis. Evidently, in the first structure, the state-dependent input $u_{\textbf{I}}$ replaces the time-dependent damping $\gamma$ in (\ref{dyn2}). While in the second structure, the feedback input $u_{\textbf{II}}$ dynamically controls the magnitude with which the potential force enters the dynamics (we assume for simplicity of exposition that $\gamma (t)=1$, however, one can modify our proposed framework and following a similar path develop the corresponding results for the case $\gamma(t)\neq 1$). 
Let $f$ be a twice differentiable function that satisfies the so-called Polyak--{\L}ojasiewicz (PL) inequality (see Assumption~(\ref{d_1})). 
Given a positive scalar $\alpha$, we seek to achieve an exponential rate of convergence $\mathcal{O}(e^{-\alpha t})$ for an unconstrained, smooth optimization problem in the suboptimality measure $f\big(X(t)\big)-f^*$. To do so, we construct the state-dependent feedback laws for each structure as follows:
\begin{align*}
u_{\textbf{I}}\big(X(t),\dot{X}(t)\big) := 
 \alpha + \frac{\| \nabla f(X(t)) \|^2 - \langle \nabla^2 f\big(X(t)\big) \dot{X}(t), \dot{X}(t) \rangle}{\langle \nabla f\big(X(t)\big), -\dot{X}(t)  \rangle},
\end{align*}
\begin{align*}
u_{\textbf{II}}\big(X(t),\dot{X}(t)\big) := 
 \frac{  \langle \nabla^2 f\big(X(t)\big) \dot{X}(t), \dot{X}(t) \rangle +(1 - \alpha) \langle \nabla f\big(X(t)\big), -\dot{X}(t)  \rangle}{\| \nabla f(X(t)) \|^2 }.
\end{align*}
Motivated by restarting schemes, we further extend the class of dynamics to hybrid control systems (see Definition \ref{def_hyb} for further details) in which both of the above ODE structures play the role of the \emph{continuous flow} in their respective hybrid dynamical extension. We next suggest an admissible control input range $[u_{\min},u_{\max}]$ that determines the \emph{flow set} of each hybrid system. Based on the model parameters $\alpha$, $u_{\min}$, and $u_{\max}$, we then construct the \emph{jump map} of each hybrid control system by the mapping $\big(X^\top,-\beta \nabla^\top f(X)\big)^\top$ guaranteeing that the range space of the jump map is contained in its respective flow set. Notice that the velocity restart schemes take the form of $\dot{X}=-\beta \nabla f(X)$. 

This paper extends the results of \cite{armanICML} in several ways which are summarized as follows:
\begin{itemize}
	\item We synthesize a state-dependent gradient coefficient  ($u_{\textbf{II}}(x)$) given a prescribed control input bound and a desired convergence rate (Theorem~\ref{theo_step_conv}). This is a complementary result to our earlier study [30] which is concerned with a state-dependent damping coefficient ($u_{\textbf{I}}(x)$). Notice that the state-dependent feature of our proposed dynamical systems differs from commonly time-dependent methodologies in the literature.

	\item We derive a lower bound on the time between two consecutive jumps for each hybrid structure. This ensures that the constructed hybrid systems admit the so-called Zeno-free solution trajectories. It is worth noting that the regularity assumptions required by the proposed structures are different (Theorems~\ref{theo_zeno} and \ref{theo_step_zeno}).
	
	\item The proposed frameworks are general enough to include a subclass of non-convex problems. Namely, the critical requirement is that the objective function $f$ satisfies the Polyak--{\L}ojasiewicz (PL) inequality (Assumption~(\ref{d_1})), which is a weaker regularity assumption than the strong convexity that is often assumed in this context for exponential convergence. 

	\item We utilize the \emph{forward-Euler} method to discretize both hybrid systems (i.e., obtain optimization algorithms). We further provide a mechanism to compute the step size such that the corresponding discrete dynamics have an exponential rate of convergence (Theorem~\ref{theo_2}). 
\end{itemize}

The remainder of this paper is organized as follows. In Section~\ref{sec:notation}, the mathematical notions are represented. The main results of the paper are introduced in Section~\ref{sec:mainres}. Section~\ref{sec:proofs} contains the proofs of the main results. We introduce a numerical example in Section~\ref{sec:examp}. This paper is finally concluded in Section~\ref{sec:conc}.

\textbf{Notations:} The sets $\mathbb{R}^n$ and $\mathbb{R}^{m\times n}$ denote the $n$-dimensional Euclidean space and the space of $m\times n$ dimensional matrices with real entries, respectively. For a matrix $M\in\mathbb{R}^{m\times n}$, $M^\top$ is the transpose of $M$, $M\succ0$ ($\prec0$) refers to $M$ positive (negative) definite, $M\succeq0$ ($\preceq0$) refers to $M$ positive (negative) semi-definite, and $\lambda_{\max}(M)$ denotes the maximum eigenvalue of $M$. The $n\times n$ identity matrix is denoted by $I_n$. For a vector $v\in\mathbb{R}^n$ and $i\in\{1,\cdots,n \}$, $v_i$ represents the $i$-th entry of $v$ and $\| v \|:=\sqrt{\Sigma_{i=1}^n~v_i^2}$ is the Euclidean 2-norm of $v$.  For two vectors $x,y\in\mathbb{R}^n$, $\langle x,y \rangle:=x^\top y$ denotes the Euclidean inner product. For a matrix $M$, $\| M \|:=\sqrt{\lambda_{\max}(M^\top M)}$ is the induced 2-norm. Given the set $S\subseteq \mathbb{R}^n$, $\partial S$ and $\text{int}(S)$ represent the boundary and the interior of $S$, respectively.

\section{Preliminaries}
\label{sec:notation}
We briefly recall some notions from hybrid dynamical systems that we will use to develop our results. 
We state the standing assumptions related to the optimization problem to be tackled in this paper. The problem statement is then introduced. 
We adapt the following definition of a hybrid control system from \cite{goebel2012hybrid} that is sufficient in the context of this paper. 
\begin{Def}[Hybrid control system]
\label{def_hyb}
A time-invariant hybrid control system $\mathcal{H}$ comprises a controlled ODE and a jump (or a reset) rule introduced as: 
\begin{equation}
\tag{$\mathcal{H}$}
\label{p1}
\left\lbrace
\begin{array}{lllc}
\dot{x} & = & F\big(x,u(x)\big), & x \in \mathcal{C}\\
x^+ & = & G(x), & \text{otherwise},
\end{array}
\right.
\end{equation}
where $x^+$ is the state of the hybrid system after a jump, the function $u:\mathbb{R}^n\rightarrow\mathbb{R}^m$ denotes a feedback signal, the function $F:\mathbb{R}^n\times\mathbb{R}^m\rightarrow\mathbb{R}^n$ is the flow map, the set $\mathcal{C}\subseteq \mathbb{R}^n$ is the flow set, and the function $G:\partial \mathcal{C}\rightarrow$ \emph{int}$(\mathcal{C})$ represents the jump map.
\end{Def}
Notice that the jump map $G(x)$ will be activated as soon as the state $x$ reaches the boundary of the flow set $\mathcal{C}$, that is $\partial \mathcal{C}$. 
In hybrid dynamical systems, the notion of \emph{Zeno behavior} refers to the phenomenon that an infinite number of jumps occur in a bounded time interval. We then call a solution trajectory of a hybrid dynamical system Zeno-free if the number of jumps within any finite time interval is bounded. 
The existence of a lower bound on the time interval between two consecutive jumps suffices to guarantee the Zeno-freeness of a solution trajectory of a hybrid control system. Nonetheless, there exist solution concepts in the literature that accept Zeno behaviors, see for example \cite{aubin2002impulse,goebel2012hybrid,goebel2006solutions,lygeros2003dynamical} and the references therein.

Consider the following class of unconstrained optimization problems:
\begin{equation}
\label{pf1}
f^*:=\underset{X\in\mathbb{R}^n}{\min} f(X),
\end{equation}
where $f:\mathbb{R}^n\rightarrow \mathbb{R}$ is an objective function. 

\begin{As}[Regularity assumptions]
\label{def_10}
We stipulate that the objective function $f:\mathbb{R}^n\rightarrow \mathbb{R}$ is twice differentiable and fulfills the following
\begin{itemize}
\item (Bounded Hessian) The Hessian of function $f$, denoted by $\nabla^2 f(x)$, is uniformly bounded, i.e.,
\begin{equation}
\tag{A1}
\label{p2}
-\ell_f I_n \preceq \nabla^2 f(x) \preceq L_f I_n,
\end{equation}
where $\ell_f$ and $L_f$ are non-negative constants.
\end{itemize}
\begin{itemize}
\item (Gradient dominated) The function $f$ satisfies the Polyak-{\L}ojasiewicz inequality with a positive constant $\mu_f$, i.e., for every $x$ in $\mathbb{R}^n$ we have
\begin{equation}
\tag{A2}
\label{d_1}
\frac{1}{2} \big\| \nabla f(x) \big\|^2 \geq \mu_f \big(f(x)-f^*\big),
\end{equation}
where $f^*$ is the minimum value of $f$ on $\mathbb{R}^n$. 
\item (Lipschitz Hessian) The Hessian of the function $f$ is Lipschitz, i.e., for every $x,y$ in $\mathbb{R}^n$ we have
\begin{align}
\tag{A3}
\label{z6}
\big\| \nabla^2 f(x) - \nabla^2 f(y) \big\| \leq H_f \| x - y  \|,
\end{align}
where $H_f$ is a positive constant.
\end{itemize}
\end{As}

We now formally state the main problem to be addressed in this paper:
\begin{Prob}
\label{prob1}
Consider the unconstrained optimization problem (\ref{pf1}) where the objective function $f$ is twice differentiable. Given a positive scalar $\alpha$, design a fast gradient-based method in the form of a hybrid control system~(\ref{p1}) with $\alpha$-exponential convergence rate, i.e. for any initial condition $X(0)$ and any $t \geq 0$ we have
\begin{equation*}
f\big(X(t)\big)-f^*\leq e^{-\alpha t} \Big(f\big(X(0)\big)-f^* \Big), 
\end{equation*}
where $\{X(t)\}_{t \geq0}$ denotes the solution trajectory of the system~(\ref{p1}). 
\end{Prob}

\begin{Rem}[Lipschitz gradient]
\label{rem_lip}
Since the function $f$ is twice differentiable, Assumption~(\ref{p2}) implies that the function $\nabla f$ is also Lipschitz with a positive constant $L_f$, i.e., for every $x, y$ in $\mathbb{R}^n$ we have
\begin{equation}
\label{p2_g}
\big\| \nabla f(x)-\nabla f(y) \big\| \leq L_f \| x-y\|.
\end{equation}
\end{Rem}

We now collect two remarks underlining some features of the set of functions that satisfy (\ref{d_1}).
\begin{Rem}[PL functions and invexity] 
The PL inequality in general does not imply the convexity of a function but rather the invexity of it. The notion of invexity was first introduced by \cite{Hanson1981}. The PL inequality (\ref{d_1}) implies that the suboptimality measure $f-f^*$ grows at most as a quadratic function of $\nabla f$. 
\end{Rem}
\begin{Rem}[Non-uniqueness of stationary points]
While the PL inequality does not require the uniqueness of the stationary points of a function (i.e., $\{x: \nabla f(x)=0 \}$), it ensures that all stationary points of the function $f$ are global minimizers \cite{CravenGlover1985}. 
\end{Rem}
We close our preliminary section with a couple of popular examples borrowed from \cite{Karimi2016}. 
\begin{Ex}[PL functions] 
The composition of a strongly convex function and a linear function satisfies the PL inequality. This class includes a number of important problems such as least squares, i.e., $f(x)=\| Ax -b \|^2 $ (obviously, strongly convex functions also satisfy the PL inequality). Any strictly convex function over a compact set satisfies the PL inequality. As such, the log-loss objective function in logistic regression, i.e., $f(x)=\Sigma_{i=1}^n\log\big(1+\text{exp}(b_ia_i^\top x)\big)$, locally satisfies the PL inequality.
\end{Ex}

\section{Main Results}
\label{sec:mainres}
In this section, the main results of this paper are provided. We begin with introducing two types of structures for the hybrid system~(\ref{p1}) motivated by the dynamics of fast gradient methods \cite{Su2016differential}. Given a positive scalar $\alpha$, these structures, indexed by \textbf{I} and \textbf{II}, enable achieving the rate of convergence $\mathcal{O}(e^{-\alpha t})$ in the suboptimality measure $f\big(x_1(t)\big)-f^*$. We then collect multiple remarks highlighting the shared implications of the two structures along with a naive type of time-discretization for these structures. The technical proofs are presented in Section~\ref{sec:proofs}. For notational simplicity, we introduce the notation $x = (x_1,x_2)$ such that the variables $x_1$ and $x_2$ represent the system trajectories $X$ and $\dot{X}$, respectively. 

\subsection{Structure \textbf{I}: state-dependent damping coefficient}
\label{sec:par_I}
The description of the first structure follows. We start with the flow map $F_{\textbf{I}}:\mathbb{R}^{2n}\times \mathbb{R}\rightarrow\mathbb{R}^{2n}$ defined as 
\begin{subequations}
\label{sH}
\begin{align}
\label{s1}
F_{\textbf{I}}\big(x,u_{\textbf{I}}(x)\big)=
\left(
\begin{aligned}
x& _2\\
-\nabla f &(x_1)
\end{aligned}
\right)+\left(
\begin{aligned}
0~&\\
-x& _2
\end{aligned}
\right)u_{\textbf{I}}(x).
\end{align}
Notice that $F_{\textbf{I}}(\cdot,\cdot)$ is the state-space representation of a 2nd-order ODE. The feedback law $u_{\textbf{I}}:\mathbb{R}^{2n}\rightarrow \mathbb{R}$ is given by
\begin{equation}
\label{s8_1}
u_{\textbf{I}}(x) = \alpha + \frac{\| \nabla f(x_1) \|^2 - \langle \nabla^2 f(x_1) x_2, x_2 \rangle}{\langle \nabla f(x_1), -x_2  \rangle}.
\end{equation}
Intuitively, the control input $u_{\textbf{I}}(x)$ is designed such that the flow map $F_{\textbf{I}}\big(x,u_{\textbf{I}}(x)\big)$ renders a level set $\sigma(t):=\langle \nabla f\big(x_1(t)\big),x_2(t) \rangle +\alpha\big(f\big(x_1(t)\big)-f^*\big)$ invariant, i.e., $\frac{d}{dt}\sigma(t)=0$. 
Next, the candidate flow set $\mathcal{C}_{\textbf{I}} \subset \mathbb{R}^{2n}$ is characterized by an admissible input interval $[\ul{u}_{\textbf{I}}~\ol{u}_{\textbf{I}}]$, i.e.,
\begin{equation}
\label{s8_2}
\mathcal{C}_{\textbf{I}} = \big\{x\in\mathbb{R}^{2n}:~ u_{\textbf{I}}(x)\in [\ul{u}_{\textbf{I}},,\ol{u}_{\textbf{I}}] \big\},
\end{equation} 
where the interval bounds $\ul{u}_{\textbf{I}},\ol{u}_{\textbf{I}}$ represent the range of admissible control values. Notice that the flow set $\mathcal{C}_{\textbf{I}}$ is the domain in which the hybrid system (\ref{p1}) can evolve continuously. Finally, we introduce the jump map $G_{\textbf{I}}:\mathbb{R}^{2n}\rightarrow\mathbb{R}^{2n}$ parameterized by a constant $\beta_{\textbf{I}}$
\begin{align}
\label{s8}
G_{\textbf{I}}(x)=\left(\begin{aligned}
x& _1  \\
-\beta_{\textbf{I}} \nabla & f(x_1)
\end{aligned}\right).
\end{align}
The parameter $\beta_{\textbf{I}}$ ensures that the range space of the jump map $G_{\textbf{I}}$ is a strict subset of $\text{int}(\mathcal{C}_{\textbf{I}})$. By construction, one can inspect that any neighborhood of the optimizer $x_1^*$ has a non-empty intersection with the flow set $\mathcal{C}_{\textbf{I}}$. That is, there always exist paths in the set $\mathcal{C}_{\textbf{I}}$ that allow the continuous evolution of the hybrid system to approach arbitrarily close to the optimizer.
\end{subequations}

We are now in a position to formally present the main results related to the structure~\textbf{I} given in \eqref{sH}. For the sake of completeness, we borrow the first result from \cite{armanICML}.
This theorem provides a framework to set the parameters $\ul{u}_{\textbf{I}}$, $\ol{u}_{\textbf{I}}$, and $\beta_{\textbf{I}}$ in (\ref{s8_2}) and (\ref{s8}) in order to ensure the desired exponential convergence rate $\mathcal{O}(e^{-\alpha t})$.
\begin{Thm}[Continuous-time convergence rate - \textbf{I}]
\label{theo_1b}
Consider a positive scalar $\alpha$ and a smooth function $f: \mathbb{R}^n\rightarrow \mathbb{R}$ satisfying Assumptions~(\ref{p2}) and (\ref{d_1}). Then, the solution trajectory of the hybrid control system (\ref{p1}) with the respective parameters (\ref{sH}) starting from any initial condition $x_1(0)$ satisfies 
\begin{equation}
\label{eqt_8b}
f\big(x_1(t)\big)-f^* \leq e^{-\alpha t} \Big( f\big(x_1(0)\big)-f^* \Big), \quad \forall t \geq0,
\end{equation}
if the scalars $\ul{u}_{\textbf{I}}$, $\ol{u}_{\textbf{I}}$, and $\beta_{\textbf{I}}$ are chosen such that
\begin{subequations}
\label{eqt_1b}
\begin{align}
\ul{u}_{\textbf{I}} & <  \alpha+  \beta_{\textbf{I}}^{-1}-L_f\beta_{\textbf{I}},  \label{eqt_1b1}\\
\label{eqt_1b2}
\ol{u}_{\textbf{I}}  & >   \alpha+  \beta_{\textbf{I}}^{-1}+\ell_f\beta_{\textbf{I}}, \\
\alpha & \leq 2 \mu_f \beta_{\textbf{I}}.  \label{eqt_1b3}
\end{align}
\end{subequations}
\end{Thm}

The next result establishes a key feature of the solution trajectories generated by the dynamics (\ref{p1}) with the respective parameters (\ref{sH}), that the solution trajectories are indeed \emph{Zeno}-free.

\begin{Thm}[Zeno-free hybrid trajectories - \textbf{I}]
\label{theo_zeno}
Consider a smooth function $f: \mathbb{R}^n\rightarrow \mathbb{R}$ satisfying Assumption~\ref{def_10}, and the corresponding hybrid control system~\eqref{p1} with the respective parameters \eqref{sH} satisfying \eqref{eqt_1b}. Given the initial condition $\Big(x_1(0),-\beta_{\textbf{I}} \nabla f\big(x_1(0)\big) \Big)$ the time between two consecutive jumps of the solution trajectory, denoted by $\tau_{\textbf{I}}$, satisfies for any scalar $r>1$
\begin{align}
\label{z_t2_03}
\tau_{\textbf{I}} \ge \log \left(\min{\bigg\{\frac{a_1}{a_2+ a_3 \big\|\nabla f\big(x_{1}(0)\big)\big\| } +1, r \bigg\}^{1/\delta}}\right), 
\end{align}
where the involved constants are defined as 
\begin{subequations}
	\begin{align}
	C & :=\frac{(\ol{u}_{\textbf{I}} - \alpha) + \sqrt{(\ol{u}_{\textbf{I}} - \alpha)^2 + 4L_f} }{2}, \label{z_t2_2}\\
	\delta & := C + \max \{\ol{u}_{\textbf{I}}, -\ul{u}_{\textbf{I}} \},  \label{z_t2_3} \\
	\mathcal{L}_f & := \max\{\ell_f, L_f \},  \label{z_t2_30} \\
	a_1 &:= \min \{ \ol{u}_{\textbf{I}}-(\alpha + \beta_{\textbf{I}}^{-1}+\ell_f \beta_{\textbf{I}}), (\alpha+\beta_{\textbf{I}}^{-1}-L_f \beta_{\textbf{I}})-\ul{u}_{\textbf{I}}  \}, \label{z_t2_00} \\
	a_2 &:= r L_f \delta^{-1}   (r \beta_{\textbf{I}} C  + 1 ) + \beta_{\textbf{I}}^{-1} + (r^2+r+1) \beta_{\textbf{I}} \mathcal{L}_f, \label{z_t2_01}\\
	a_3  &:= r^3 \beta_{\textbf{I}}^2 H_f \delta^{-1}. \label{z_t2_02}
	\end{align}
\end{subequations}
Consequently, the solution trajectories are Zeno-free.
\end{Thm}
\begin{Rem}[Non-uniform inter-jumps - \textbf{I}]
\label{rem_41}
Notice that Theorem~\ref{theo_zeno} suggests a lower-bound for the inter-jump interval $\tau_{\textbf{I}}$ that depends on $\| \nabla f\big(x_1\big)\|$. In light of the fact that the solution trajectories converge to the optimal solutions, and as such $\nabla f\big(x_1\big)$ tends to zero, one can expect that the frequency at which the jumps occur reduces as the hybrid control system evolves in time.
\end{Rem}

\subsection{Structure \textbf{II}: state-dependent potential coefficient}
\label{sec:par_II}

\begin{subequations}
\label{sH_step}
In this subsection, we first provide the structure~\textbf{II} for the hybrid control system~\eqref{p1}. We skip the the details of differences with the structure~\textbf{I} and differ it to Subection~\ref{sec:par_I} and Section~\ref{sec:proofs}. Consider the flow map $F_{\textbf{II}}:\mathbb{R}^{2n}\times \mathbb{R}\rightarrow\mathbb{R}^{2n}$ given by
\begin{align}
\label{step_01}
F_{\textbf{II}}\big(x,u_{\textbf{II}}(x)\big)=
\left(
\begin{aligned}
~x&_2\\
-&x_2
\end{aligned}
\right)+\left(
\begin{aligned}
0~&\\
-\nabla f &(x_1)
\end{aligned}
\right)u_{\textbf{II}}(x),
\end{align}
and the feedback law $u_{\textbf{II}}:\mathbb{R}^{2n}\rightarrow \mathbb{R}$ given by
\begin{align}
\label{step_02}
u_{\textbf{II}} (x) = \frac{ \langle \nabla^2 f(x_1) x_2, x_2 \rangle  + ( 1- \alpha) \langle \nabla f(x_1), - x_2 \rangle  }{ \| \nabla f(x_1) \|^2 }.
\end{align}
Notice that here the input $u_{\textbf{II}} (x)$ is derived along the same lines as in structure~\textbf{I}. 
The feedback input $u_{\textbf{II}} (x)$ is synthesized such that the level set $\sigma(t):=\langle \nabla f\big(x_1(t)\big),x_2(t) \rangle +\alpha\big(f\big(x_1(t)\big)-f^*\big)$ remains constant as the dynamics $x$ evolve based on the flow map $F_{\textbf{II}}\big(x,u_{\textbf{II}}(x)\big)$. 
The candidate flow set $\mathcal{C}_{\textbf{II}} \subset \mathbb{R}^{2n}$ is parameterized by an admissible interval $[\ul{u}_{\textbf{II}}~\ol{u}_{\textbf{II}}]$ as follows:
\begin{align}
\label{step_03}
\mathcal{C}_{\textbf{II}} = \left\lbrace x\in \mathbb{R}^{2n}:~ u_{\textbf{II}}(x) \in [\ul{u}_{\textbf{II}},\ol{u}_{\textbf{II}}] \right\rbrace.
\end{align}
Parameterized in a constant $\beta_{\textbf{II}}$, the jump map $G_{\textbf{II}}:\mathbb{R}^{2n}\rightarrow\mathbb{R}^{2n}$ is given by
\begin{align}
\label{step_04}
G_{\textbf{II}}(x)= \left(
\begin{aligned}
 x&_1 \\
-\beta_{\textbf{II}} \nabla & f(x_1)
\end{aligned}
\right).
\end{align}
\end{subequations}
\begin{Thm}[Continuous-time convergence rate - \textbf{II}]
\label{theo_step_conv}
Consider a positive scalar $\alpha$ and a smooth function $f: \mathbb{R}^n\rightarrow \mathbb{R}$ satisfying Assumptions~(\ref{p2}) and (\ref{d_1}). Then, the solution trajectory of the hybrid control system~(\ref{p1}) with the respective parameters~\eqref{sH_step} starting from any initial condition $x_1(0)$ satisfies the inequality~\eqref{eqt_8b} if the scalars $\ul{u}_{\textbf{II}}$, $\ol{u}_{\textbf{II}}$, and $\beta_{\textbf{II}}$ are chosen such that
\begin{subequations}
\label{step_05}
\begin{align}
\ul{u}_{\textbf{II}} & < -\ell_f \beta_{\textbf{II}}^2 + (1 - \alpha) \beta_{\textbf{II}}, \label{step_05_2}  \\
\ol{u}_{\textbf{II}} & > L_f \beta_{\textbf{II}}^2 + (1 - \alpha) \beta_{\textbf{II}}, \label{step_05_1} \\
\alpha & \leq 2 \mu_f \beta_{\textbf{II}}. \label{step_05_3} 
\end{align}
\end{subequations}
\end{Thm}

\begin{Thm}[Zeno-free hybrid trajectories - \textbf{II}]
\label{theo_step_zeno}
Consider a smooth function $f: \mathbb{R}^n\rightarrow \mathbb{R}$ satisfying Assumptions~\eqref{p2} and \eqref{d_1}, and the hybrid control system~\eqref{p1} with the respective parameters \eqref{sH_step} satisfying \eqref{step_05}. Given the initial condition $\Big(x_1(0),-\beta_{\textbf{II}} \nabla f\big(x_1(0)\big) \Big)$ the time between two consecutive jumps of the solution trajectory, denoted by $\tau_{\textbf{II}}$, satisfies for any scalar $r \in (0,1)$
\begin{align}
\label{step_3}
\tau_{\textbf{II}} \geq \min\left\{ r\omega^{-1}, \delta ( b_1+b_2)^{-1}  \right\}.
\end{align}
where the involved scalars are defined as
\begin{align*}
\delta &:=\min\big\{\ol{u}_{\textbf{II}}-(L_f\beta_{\textbf{II}}^2+(1-\alpha)\beta_{\textbf{II}}), (-\ell_f\beta_{\textbf{II}}^2+(1-\alpha)\beta_{\textbf{II}})-\ul{u}_{\textbf{II}} \big\}, \\
U &:=\max \{\ol{u}_{\textbf{II}}, -\ul{u}_{\textbf{II}}\},\\
\mathcal{L}_f &:=\max \{ \ell_f, L_f \},\\
\omega &:=\mathcal{L}_f (\beta_{\textbf{II}}^2+\beta_{\textbf{II}} U)^{\frac{1}{2}},\\
b1 & := \frac{2 \mathcal{L}_f \beta_{\textbf{II}} \big(  U +  \omega  (\beta_{\textbf{II}}  +U)   \big) }{(1-r)^3}, \\
b_2& := |\alpha - 1| \frac{ 2 \omega  \beta_{\textbf{II}} }{(1-r)^3} + |\alpha - 1| \alpha \beta_{\textbf{II}} (1+r).
\end{align*}
Thus, the solution trajectories are Zeno-free.
\end{Thm}

\begin{Rem}[Uniform inter-jumps - \textbf{II}]
\label{rem_41}
Notice that unlike Theorem~\ref{theo_zeno}, the derived lower-bound for the inter-jump interval $\tau_{\textbf{II}}$ is uniform in the sense that the bound is independent of $\| \nabla f\big(x_1\big)\|$. Furthermore, the regularity requirement on the function $f$ is weaker than the one used in Theorem~\ref{theo_zeno}, i.e., the function $f$ is not required to satisfy the Assumption~\eqref{z6}.  
\end{Rem}

Notice that the main differences between the structures~\eqref{sH}, \eqref{sH_step} lie in the flow maps and the feedback laws. On the other hand, these structures share the key feature of enabling an $\alpha$-exponential convergence rate for the hybrid system \eqref{p1} through their corresponding control inputs. The reason explaining the aforementioned points is deferred until later in Section~\ref{sec:proofs}.

\subsection{Further Discussions}
\label{sec:par_I}

In what follows, we collect several remarks regarding the common features of the proposed structures. Then, we apply the \emph{forward-Euler} method of time-discretization to these structures of the hybrid control system~\eqref{p1}. The proposed discretizations guarantee an exponential rate of convergence in the suboptimality measure $f(x_1^k)-f^*$, where $k$ is the iteration index. 

\begin{Rem}[Weaker regularity than strong convexity]
\label{rem_2}
The PL inequality is a weaker requirement than strong convexity. 
Notice that although the class of functions that satisfy the PL inequality are in general non-convex, the set of minimizers of such functions should still be a convex set. 
\end{Rem}

\begin{Rem}[Hybrid embedding of restarting]
\label{rem_3}
The hybrid frameworks intrinsically capture restarting schemes through the jump map. The schemes are a weighted gradient where the weight factor $\beta_{\textbf{I}}$ or $\beta_{\textbf{II}}$ is essentially characterized by the given data $\alpha$, $\mu_f$, $\ell_f$, and $L_f$. One may inspect that the constant $\beta_{\textbf{I}}$ or $\beta_{\textbf{II}}$ can be in fact introduced as a state-dependent weight factor to potentially improve the performance. Nonetheless, for the sake of simplicity of exposition, we do not pursue this level of generality in this paper. 
\end{Rem}

\begin{Rem}[2nd-order information]
\label{rem_40}
Although our proposed frameworks require 2nd-order information, i.e., the Hessian $\nabla^2 f$, this requirement only appears in a mild form as an evaluation in the same spirit as the modified Newton step proposed in \cite{nesterov2006cubic}. Furthermore, we emphasize that our results still hold true if one replaces $\nabla^2 f(x_1)$ with its upper-bound $L_f I_n$ following essentially the same analysis. For further details we refer the reader to the proof of Theorem~\ref{theo_step_conv}.
\end{Rem}

\begin{Rem}[Fundamental limits on control input]
\label{rem_step_input}
An implication of Theorem \ref{theo_step_conv} is that if the desired convergence rate $\alpha > \big(\frac{2\mu_f}{2\mu_f+\ell_f}\big)$, it is then required to choose $\ul{u}_{\textbf{II}}<0$, indicating that the system may need to receive energy through a negative damping. On a similar note, Theorem~\ref{theo_1b} asserts that the upper bound requires $\ol{u}_{\textbf{I}}>\alpha$, and if $\alpha > \big(\frac{2 \mu_f}{\sqrt{\max\{L_f-2\mu_f,0\}}}\big)$, we then have to set $\ul{u}_{\textbf{I}}<0$ \cite[Remark~3.4]{armanICML}.
\end{Rem}

\subsection{Discrete-Time Dynamics}
In the next result, we show that if one applies the forward-Euler method on the two proposed structures properly, the resulting discrete-time hybrid control systems possess exponential convergence rates. Suppose $i\in\{\textbf{I},\textbf{II}\}$ and let us denote by $s$ the time-discretization step size. Consider the discrete-time hybrid control system
\begin{equation}
\label{d1}
x^{k+1} =\left\lbrace
\begin{array}{lc}
F_{d,i}\big(x^k,u_{d,i}(x^k)\big), & x^k \in \mathcal{C}_{d,i}\\
 G_{d,i}(x^k), & \text{otherwise},
\end{array}
\right.
\end{equation}
where $F_{d,i}$, $G_{d,i}$, and $\mathcal{C}_{d,i}$ are the flow map, the jump map, and the flow set, respectively. The discrete flow map $F_{d,i}:\mathbb{R}^{2n}\times \mathbb{R}\rightarrow\mathbb{R}^{2n}$ is given by 
\begin{subequations}
\label{dHd}
\begin{align}
\label{d1_1}
F_{d,i}\big(x^k,u_{d,i}(x^k) \big)=x^k+sF_i\big(x^k,u_i(x^k)\big),\; i \in \{ \mathbf{I}, \mathbf{II} \},
\end{align} 
where $F_i$ and $u_i$ are defined in \eqref{s1} and \eqref{s8_1}, or \eqref{step_01} and \eqref{step_02} based on the considered structure~$i$. The discrete flow set $\mathcal{C}_{d,i}\subset \mathbb{R}^{2n}$ is defined as
\begin{align}
\label{d2_1}
\mathcal{C}_{d,i} =  \big\{(x_1^k,x_2^k)\in\mathbb{R}^{2n}:
 c_1 \| x^k_2 \|^2 \leq \| \nabla f(x^k_1)  \|^2 \leq c_2 \langle \nabla f(x^k_1), -x^k_2 \rangle   \big\},
\end{align}
and, $c_1$ and $c_2$ are two positive scalars. The discrete jump map $G_{d,i}:\mathbb{R}^{2n}\rightarrow\mathbb{R}^{2n}$ is given by $G_{d,i}(x^k)=\big((x^k)^\top,-\beta \nabla^\top f(x^k)\big)^\top$.
\end{subequations}

It is evident in the flow sets $\mathcal{C}_{d,i}$ of the discrete-time dynamics that these sets are no longer defined based on admissible input intervals. The reason has to do with the difficulties that arise from appropriately discretizing the control inputs $u_{\textbf{I}}$ and $u_{\textbf{II}}$. Nonetheless, the next result guarantees exponential rate of convergence of the discrete-time control system~\eqref{d1} with either of the respective structure \textbf{I} or \textbf{II}, by introducing a mechanism to set the scalars $c_1$, $c_2$, and $\beta$.

\begin{Thm}[Stable discretization - \textbf{I \& II}]
\label{theo_2}
Consider a smooth function $f: \mathbb{R}^n\rightarrow \mathbb{R}$ satisfying  Assumptions~\eqref{p2} and \eqref{d_1}. The solution trajectory of the discrete-time hybrid control system~(\ref{d1}) with the respective structure $i \in \{\mathbf{I},\mathbf{II}\}$ and starting from any initial condition $x_1^0$, satisfies
\begin{align}
\label{d4}
f(x_1^{k+1})-f^* \leq \lambda(s,c_1,c_2,\beta) \big( f(x_1^k) - f^*\big), 
\end{align}
with $\lambda(s,c_1,c_2,\beta)\in (0,1)$ given by
\begin{align}
\label{d_r}
\lambda(s,c_1,c_2,\beta):=1+ 2 \mu_f \big( - \frac{s}{c_2} + \frac{L_f}{2 c_1} s^2 \big),
\end{align}
if the parameters $s$, $c_1$ ,$c_2$, and $\beta$ satisfy
\begin{subequations}
\label{d4_d}
\begin{align}
& \sqrt{c_1} \leq c_2, \label{d4_d1}\\
&\beta^2 c_1 \leq 1 \leq \beta c_2 ,\label{d4_d2}\\
&c_2 L_f s < 2 c_1.
\end{align}
\end{subequations}
\end{Thm}

\begin{Rem}[Naive discretization]
\label{rem_5}
We would like to emphasize that the exponential convergence of the proposed discretization method solely depends on the dynamics $x_1$ and the properties of the objective function $f$. Thus, we deliberately avoid labeling the scalars $c_1$, $c_2$, and $\beta$ by the structure index $i$. Crucially, the structures of the control laws do not impact the relations~\eqref{d4_d} in Theorem~\ref{theo_2}, see Subsection~\ref{subsec:proof_2} for more details. In light of the above facts, we believe that a more in-depth analysis of the dynamics along with the control structures may provide a more intelligent way to improve the discretization result of Theorem~\ref{theo_2}.
\end{Rem}

\begin{Cor}[Optimal guaranteed rate]
\label{Cor_1}
The optimal convergence rate guaranteed by Theorem~\ref{theo_2} for the discrete-time dynamics is $\lambda^*:=\big(1-\frac{\mu_f}{L_f} \big)$ and
\begin{align*}
&\sqrt{c_1^*}  = c^*_2 =  \frac{1}{\beta^*} = L_f s^*.
\end{align*}
\end{Cor}
The pseudocode to implement the above corollary is presented in Algorithm~\ref{alg:example} using the discrete-time dynamics~\eqref{d1} with the respective parameters \textbf{I} or \textbf{II}.
\begin{algorithm}[tb]
   \caption{Sate-dependent fast gradient method}
   \label{alg:example}
\begin{algorithmic}
   \STATE {\bfseries Input:} data $x_1^0$, $\ell_f$, $L_f$, $\mu_f$, $\alpha \in \mathbb{R}^+$, $k_{\max} \in \mathbb{N}^+$, $i\in\{ \mathbf{I}, \mathbf{II} \}$
   \STATE {\bfseries Set:} $\sqrt{c_1}  = c_2 =  \beta^{-1} = L_f s$, $x_2^0=-\beta \nabla f(x_1^0)$\\
   $\quad\quad\;\; x^0=(x_1^0,x_2^0)$
   \FOR{$k=1$ {\bfseries to} $k_{\max}$}
   \IF{$c_1 \| x^k_2 \|^2 \leq \| \nabla f(x^k_1)  \|^2 \leq c_2 \langle  \nabla f(x^k_1),-x^k_2 \rangle$}
   \STATE $x^{k+1} \leftarrow F_{d,i}(x^k)$
   \ELSE
   \STATE $x^{k+1} \leftarrow G_{d,i}(x^k)$
   \ENDIF
   \ENDFOR
\end{algorithmic}
\end{algorithm}

\begin{Rem}[Gradient-descent rate matching]
Notice that the rate $1-\frac{\mu_f}{L_f}$ in Corollary~\ref{Cor_1} is equal to the rate guaranteed by the gradient descent method for functions that satisfy the PL inequality~\eqref{d_1}, see e.g., \cite{Karimi2016}. This is in fact another inefficiency indicator of a straightforward application of the forward-Euler method on the continuous-time hybrid control systems that are proposed in this paper. Moreover, it is worth emphasizing that Nesterov's fast method achieves the optimal rate $1-\sqrt{\frac{\sigma_f}{L_f}}$ for strongly convex functions with the strong convexity constant $\sigma_f$ \cite{Nesterov2004}.
\end{Rem}

\section{Technical Proofs}
\label{sec:proofs}

\subsection{Proof of Theorem~\ref{theo_zeno}}
\label{subsec:pro_zeno1}
In this subsection, we first set the stage by providing two intermediate results regarding the properties of dynamics of the hybrid control system~(\ref{p1}) with the respective parameters (\ref{sH}). We then employ these facts to formally state the proof of Theorem~\ref{theo_zeno}. The next lemma reveals a relation between $\nabla f(x_1)$ and $x_2$ along the trajectories of the hybrid control system. In this subsection, for the sake of brevity we denote $x_1(t)$ and $x_1(0)$ by $x_1$ and $x_{1,0}$, respectively. We adapt the same change of notation for $x_2$ and $x$, as well.
\begin{Lem}[Velocity lower bound]
\label{lem_z1}
Consider the continuous-time hybrid control system (\ref{p1}) with the respective parameters (\ref{sH}) satisfying (\ref{eqt_1b}) where the function $f$ satisfies Assumptions~(\ref{p2}) and (\ref{d_1}). Then, we have
\begin{equation}
\begin{aligned}
\label{z1}
\big\| \nabla f(x_1) \big\| \leq C  \| x_2 \| , 
\end{aligned}
\end{equation}
where $C$ is given by (\ref{z_t2_2}).
\end{Lem}
\begin{proof}
Notice that, by the definition of the control law and the upper bound condition $u_{\textbf{I}}(x) \leq \ol{u}_{\textbf{I}}$, we have
\begin{align*}
 \big\| \nabla f(x_1) \big\|^2 - \langle \nabla^2 f(x_1) x_2 , x_2 \rangle 
 \leq  (\ol{u}_{\textbf{I}} - \alpha) \langle \nabla f(x_1), -x_2 \rangle \leq 
 (\ol{u}_{\textbf{I}} - \alpha) \big\| \nabla f(x_1) \big\| \cdot \| x_2 \|,
\end{align*}
where the second inequality follows from the Cauchy-Schwarz inequality. Since the function $f$ satisfies Assumption~(\ref{p2}), one can infer that
\begin{align*}
\big\| \nabla f(x_1) \big\|^2 - L_f \| x_2 \|^2 \leq (\ol{u}_{\textbf{I}} - \alpha) \big\| \nabla f(x_1) \big\| \cdot \| x_2\|,
\end{align*}
which in turn can be reformulated into
\begin{align}
\label{z3}
\frac{\big\| \nabla f(x_1) \big\|^2}{\| x_2 \|^2} - (\ol{u}_{\textbf{I}} - \alpha) \frac{\big\| \nabla f(x_1) \big\|}{\| x_2 \|} - L_f \leq 0.
\end{align}
Defining the variable $y:= \big\| \nabla f(x_1) \big\| / \| x_2 \|$, the inequality~(\ref{z3}) becomes the quadratic inequality $y^2 - (\ol{u}_{\textbf{I}} - \alpha) y -L_f \leq 0$. Taking into account that $y\geq 0$, it then follows from (\ref{z1}) that
\begin{align*}
y=\frac{\big\| \nabla f(x_1) \big\|}{\| x_2 \|} \leq \frac{(\ol{u}_{\textbf{I}} - \alpha) + \sqrt{(\ol{u}_{\textbf{I}} - \alpha)^2 + 4L_f} }{2}  =: C.
\end{align*}
This concludes the proof of Lemma~\ref{lem_z1}.
\end{proof}

In the following, we provide a result that indicates the variation of norms $x_1$ and $x_2$, along the trajectories of the hybrid control system, are bounded in terms of time while they evolve according to the continuous mode. Since the hybrid control system is time-invariant, such bounds can be generalized to all inter-jump intervals.
\begin{Lem}[Trajectory growth rate]
\label{lem_z2}
Suppose that the same conditions as specified in Lemma~\ref{lem_z1} hold, and the hybrid control system (\ref{p1}), (\ref{sH}) starts from the initial condition $\big(x_{1,0},-\beta_{\textbf{I}} \nabla f(x_{1,0}) \big)$ for some $x_{1,0} \in \mathbb{R}^n$. Then
\begin{subequations}
\label{z4_m}
\begin{align}
\label{z4_0}
& \| x_1 - x_{1,0} \| \leq \delta^{-1} \|  x_{2,0}  \| \big( e^{\delta t} -1 \big),\\
\label{z4_1}
&\| x_2 - x_{2,0} \| \leq  \| x_{2,0} \| \big( e^{\delta t} - 1  \big),
\end{align}
\end{subequations}
where $\delta$ is given by (\ref{z_t2_3}).
\end{Lem}
\begin{proof} 
Using the flow dynamics~(\ref{s1}) we obtain
\begin{equation}
\label{z4_30} 
\begin{aligned}
\frac{d}{dt} \| x_2 \| \leq \Big\| \frac{d}{dt} x_2 \Big\| \leq \big\| \nabla f( x_1 ) \big\| + \big| u_{\textbf{I}} (x) \big| \cdot \| x_2 \|  
 \leq (C + \max\{\ol{u}_{\textbf{I}},-\ul{u}_{\textbf{I}}\} ) \| x_2 \| = \delta \| x_2 \|.
\end{aligned}
\end{equation}
The inequality (\ref{z4_30}) implies that
\begin{align}
\label{z4_3}
\| x_2 \| \leq \| x_{2,0} \| e^{\delta t}.
\end{align}
Furthermore, notice that
\begin{align*}
\frac{d}{dt} \| x_1 - x_{1,0} \| &\leq \Big\| \frac{d}{dt} ( x_1 - x_{1,0} )  \Big\| = \|  x_2 \|.
\end{align*}
Integrating the two sides of the above inequality leads to
\begin{align*}
 \| x_1 - x_{1,0} \| \leq \int_0^{t}~ \big\|  x_2(s)  \big\| ~ds \leq \int_0^{t}~ \|  x_{2,0} \| e^{\delta s} ~ds
  = \frac{\|  x_{2,0} \|}{\delta} \big( e^{\delta t} -1 \big),
\end{align*}
in which we made use of (\ref{z4_3}). Hence, the inequality~(\ref{z4_0}) in Lemma~\ref{lem_z1} is concluded. Next, we shall establish the inequality~(\ref{z4_1}). Note that
\begin{equation*}
\begin{aligned}
\frac{d}{dt} \| x_2 - x_{2,0} \|  \leq \Big\| \frac{d}{dt} ( x_2 - x_{2,0} )  \Big\| = \Big\| \frac{d}{dt} x_2  \Big\| \leq \delta \big\| x_2 \big\|
 \leq \delta \| x_2 -x_{2,0} \| + \delta \| x_{2,0} \|.
\end{aligned}
\end{equation*}
Applying Grownwall's inequality \cite[Lemma~A.1]{khalil1996noninear} then leads to the desired inequality~(\ref{z4_1}). The claims in Lemma~\ref{lem_z2} follow.
\end{proof}

\textbf{Proof of Theorem~\ref{theo_zeno}:} The proof comprises five steps, and the key part is to guarantee that during the first inter-jump interval the quantity $\big|u_{\textbf{I}}( x ) - u_{\textbf{I}}( x_{,0} )\big|$ is bounded by a continuous function $\phi\Big(t,\big\|\nabla f(x_{1,0})\big\|\Big)$, which is exponential in its first argument and linear in its second argument. Then, it follows from the continuity of the function $\phi$ that the solution trajectories of the hybrid control system are Zeno-free.  

\textbf{Step 1:} Let us define $g(t):=\langle \nabla f(x_1), -x_2  \rangle$. We now compute the derivative of $g(t)$ along the trajectories of the hybrid control system~(\ref{p1}),~(\ref{sH}) during the first inter-jump interval, i.e.,
\begin{equation*}
\begin{aligned}
 \frac{d}{dt} g(t)& 
  = \langle \nabla^2 f( x_1) x_2, -x_2 \rangle + \langle \nabla f( x_1), u_{\textbf{I}} ( x ) x_2 + \nabla f( x_1 ) \rangle \\
&=  - \langle \nabla^2 f( x_1) x_2, x_2 \rangle + \big\| \nabla f( x_1) \big\|^2 + u_{\textbf{I}} ( x) \langle \nabla f( x_1), x_2  \rangle \\
& = - \alpha \langle \nabla f(x_1), -x_2  \rangle = - \alpha ~ g(t).
\end{aligned}
\end{equation*}
According to the above discussion and considering the initial state $x_{2,0}=- \beta_{\textbf{I}} \nabla f(x_{1,0})$, it follows that 
\begin{align}
\label{z5}
\langle \nabla f(x_1), -x_2 \rangle = \beta_{\textbf{I}} \big\| \nabla f( x_{1,0}) \big\|^2 e ^{-\alpha t}.
\end{align}

\textbf{Step 2:} The quantity $\Big| e^{\alpha t} \big\| \nabla f( x_1)  \big\|^2 - \big\| \nabla f( x_{1,0})  \big\|^2 \Big|$ is bounded along the trajectories of the hybrid control system~(\ref{p1}) with the respective parameters (\ref{sH}) during the first inter-jump interval, i.e.,
\begin{equation*}
\label{z7}
\begin{aligned}
 \Big| e^{\alpha t} \big\| \nabla f( x_1 )  \big\|^2 - \big\| \nabla f( x_{1,0})  \big\|^2 \Big|  
& = \Big| e^{\alpha t} \big\| \nabla f( x_1)  \big\|^2 - (e^{\alpha t}-e^{\alpha t} + 1) \big\| \nabla f( x_{1,0})  \big\|^2 \Big| \\
& \overset{\text{(i)}}{\leq} e^{\alpha t}  \Big|\big\| \nabla f( x_1)  \big\|^2 - \big\| \nabla f( x_{1,0})  \big\|^2 \Big|    + (e^{\alpha t} - 1) \big\| \nabla f( x_{1,0})  \big\|^2 \\
& = e^{\alpha t}  \Big|\big\langle \nabla f( x_1) -  \nabla f( x_{1,0}),\nabla f( x_1) +  \nabla f( x_{1,0})  \big\rangle \Big|   \\
 &\qquad\quad  + (e^{\alpha t} - 1) \big\| \nabla f( x_{1,0})  \big\|^2 \\
& \overset{\text{(ii)}}{\leq} e^{\alpha t}  \big\| \nabla f( x_1) -  \nabla f( x_{1,0})\big\|\cdot \big\| \nabla f( x_1) +  \nabla f( x_{1,0})  \big\|  \\
 & \qquad\quad  + (e^{\alpha t} - 1) \big\| \nabla f( x_{1,0})  \big\|^2 \\
& \overset{\text{(iii)}}{\leq} e^{\alpha t} L_f  \|  x_1 -  x_{1,0} \| \cdot 
 \big(\beta_{\textbf{I}} C e^{\delta t} + 1 \big) \frac{\| x_{2,0} \|}{\beta_{\textbf{I}}}    + \big(e^{\alpha t} - 1\big) \frac{\| x_{2,0} \|^2}{\beta_{\textbf{I}}^2}  \\
& \overset{\text{(iv)}}{\leq} e^{\alpha t} L_f \big(e^{\delta t} - 1 \big)  \frac{\| x_{2,0} \|}{\delta} \cdot 
 \big(\beta_{\textbf{I}} C e^{\delta t} + 1 \big) \frac{\| x_{2,0} \|}{\beta_{\textbf{I}}}    + \big(e^{\alpha t} - 1\big) \frac{\| x_{2,0} \|^2}{\beta_{\textbf{I}}^2} \\
&= \left( \frac{L_f}{\delta \beta_{\textbf{I}}}  e^{\alpha t} \big(\beta_{\textbf{I}} C e^{\delta t} + 1 \big) \big(e^{\delta t} - 1 \big) +
\frac{1}{\beta_{\textbf{I}}^2} \big(e^{\alpha t} - 1\big) \right)    \|  x_{2,0}\|^2,
\end{aligned}
\end{equation*}
where we made use of the triangle inequality in the inequality (i), the Cauchy-Schwarz inequality in the inequality (ii), Assumption~(\ref{p2}) and its consequence in Remark~\ref{rem_lip} along with the triangle inequality in the inequality (iii), and the inequality (\ref{z4_0}) in the inequality (iv), respectively.

\textbf{Step 3:} 
Observe that
\begin{equation*}
\label{z8}
\begin{aligned}
& \big| e^{\alpha t} \langle \nabla^2 f(x_1) x_2,x_2 \rangle 
- \langle \nabla^2 f(x_{1,0}) x_{2,0},x_{2,0}  \rangle \big| \\
& \quad = \Big| e^{\alpha t} \big\langle \big[\nabla^2 f(x_1)-\nabla^2 f(x_{1,0})+\nabla^2 f(x_{1,0})\big] x_2,x_2  \big\rangle  - \big(e^{\alpha t}-e^{\alpha t} + 1\big) \langle \nabla^2 f(x_{1,0}) x_{2,0},x_{2,0} \rangle \Big| \\
&\quad = \Big| e^{\alpha t} \big\langle  \big[\nabla^2 f(x_1)-\nabla^2 f(x_{1,0})\big] x_2,x_2  \big\rangle  + e^{\alpha t} \langle \nabla^2 f(x_{1,0}) x_2,x_2  \rangle  - e^{\alpha t} \langle \nabla^2 f(x_{1,0}) x_{2,0},x_{2,0} \rangle \\
&\qquad\quad + \big(e^{\alpha t}- 1 \big) \langle \nabla^2 f(x_{1,0}) x_{2,0},x_{2,0} \rangle \Big| \\
&\quad \overset{\text{(i)}}{\leq} e^{\alpha t} \Big| \big\langle  \big[\nabla^2 f(x_1)-\nabla^2 f(x_{1,0})\big] x_2,x_2  \big\rangle \Big|  + e^{\alpha t} \Big| \langle \nabla^2 f(x_{1,0}) x_2,x_2  \rangle - \langle \nabla^2 f(x_{1,0}) x_{2,0},x_{2,0}  \rangle \Big| \\
&\qquad\quad + \big(e^{\alpha t}- 1 \big) \Big| \langle \nabla^2 f(x_{1,0}) x_{2,0},x_{2,0} \rangle \Big| \\
&\quad \overset{\text{(ii)}}{\leq} e^{\alpha t} H_f \| x_1 - x_{1,0} \| \cdot \| x_2 \|^2  + e^{\alpha t} \Big|  \big\langle \nabla^2 f(x_{1,0}) \big[ x_2 - x_{2,0}\big], x_2 + x_{2,0} \big\rangle \Big|   + \mathcal{L}_f \| x_{2,0} \|^2 \big(e^{\alpha t} - 1 \big),
\end{aligned}
\end{equation*}
where the inequality (i) follows from the triangle inequality, and the inequality (ii) is an immediate consequence of Assumptions~\eqref{z6} and \eqref{p2}, recalling $\mathcal{L}_f=\max\{\ell_f,L_f \}$. According to the above analysis, one can deduce that
\begin{equation*}
\label{z8}
\begin{aligned}
& \big| e^{\alpha t} \langle \nabla^2 f(x_1) x_2,x_2 \rangle 
- \langle \nabla^2 f(x_{1,0}) x_{2,0},x_{2,0}  \rangle \big| \\
& \overset{\text{(i)}}{\leq} e^{\alpha t} H_f \frac{\| x_{2,0} \|}{\delta} \big(e^{\delta t} - 1 \big) \cdot e^{2\delta t} \|  x_{2,0} \|^2  + e^{\alpha t} \mathcal{L}_f \| x_2 - x_{2,0}  \| \cdot \| x_2 + x_{2,0} \|  + \big( e^{\alpha t} -1 \big) \mathcal{L}_f \| x_{2,0} \|^2 \\
& \overset{\text{(ii)}}{\leq}  \frac{ H_f }{\delta} e^{(\alpha+2\delta) t}  \big\|  x_2(0) \big\|^3 \cdot (e^{\delta t} - 1)   + e^{\alpha t} \mathcal{L}_f \big(e^{\delta t} - 1\big) \| x_{2,0} \| \cdot \big(e^{\delta t} + 1\big) \| x_{2,0} \|   + \mathcal{L}_f \| x_{2,0} \|^2 \big( e^{\alpha t} -1 \big) \\
& =  \Big(
( H_f / \delta) ~ e^{(\alpha+2\delta) t}  \|  x_{2,0} \| \cdot \big( e^{\delta t} - 1 \big)   + \mathcal{L}_f  \big( e^{(\alpha+\delta) t} + e^{\alpha t} \big) \big( e^{\delta t} - 1 \big)  + \mathcal{L}_f (e^{\alpha t} -1) \Big) \| x_{2,0} \|^2,
\end{aligned}
\end{equation*}
where we made use of the inequality (\ref{z4_0}), the inequality (\ref{z4_1}), and the triangle inequality in the inequality (i), and the inequality (\ref{z4_1}) and the triangle inequality in the inequality~(ii), respectively.

\textbf{Step 4:} We now study the input variation $\big| u_{\textbf{I}}( x) - u_{\textbf{I}}( x_{,0} )  \big|$ along the solution trajectories of the hybrid control system~\eqref{p1}, \eqref{sH} during the first inter-jump interval. Observe that 
\begin{equation*}
\label{z9}
\begin{aligned}
& \big|u_{\textbf{I}}( x ) - u_{\textbf{I}}( x_{,0} )\big| \\
&\quad = \Big| \frac{\big\| \nabla f(x_1)  \big\|^2 - \langle \nabla^2 f(x_1) x_2(t),x_2  \rangle}
{\langle \nabla f(x_1), -x_2 \rangle}   - \frac{\big\| \nabla f( x_{1,0})  \big\|^2 - \langle \nabla^2 f(x_{1,0}) x_{2,0},x_{2,0}  \rangle}
{\langle \nabla f(x_{1,0}), -x_{2,0}  \rangle}
\Big|\\
&\quad = \Big| \frac{\big\| \nabla f(x_1)  \big\|^2 }
{\beta_{\textbf{I}} \big\| \nabla f(x_{1,0}) \big\|^2 e ^{-\alpha t}} -\frac{ \langle \nabla^2 f(x_1) x_2 ,x_2  \rangle}
{\beta_{\textbf{I}} \big\| \nabla f( x_{1,0}) \big\|^2 e ^{-\alpha t}}  - \frac{ \big\| \nabla f(x_{1,0})  \big\|^2 }
{\beta_{\textbf{I}} \big\| \nabla f( x_{1,0}) \big\|^2} 
+ \frac{ \langle \nabla^2 f(x_{1,0}) x_{2,0}, x_{2,0}  \rangle}
{\beta_{\textbf{I}} \big\| \nabla f( x_{1,0}) \big\|^2}
\Big| \\
&\quad \overset{\text{(i)}}{\leq} \frac{1}{\beta_{\textbf{I}} \big\| \nabla f( x_{1,0}) \big\|^2} \Big| e^{\alpha t} \big\| \nabla f( x_1)  \big\|^2 - \big\| \nabla f(x_{1,0})  \big\|^2 \Big| \\
& \qquad  + 
\frac{1}{\beta_{\textbf{I}} \big\| \nabla f(x_{1,0}) \big\|^2}   \Big| e^{\alpha t} \big\langle \nabla^2 f(x_1) x_2 , x_2  \big\rangle 
- \langle \nabla^2 f(x_{1,0}) x_{2,0}, x_{2,0}  \rangle \Big| \\
&\quad \overset{\text{(ii)}}{=} \frac{\beta_{\textbf{I}}}{ \| x_{2,0} \|^2} \Big| e^{\alpha t} \big\| \nabla f(x_1)  \big\|^2 - \big\| \nabla f(x_{1,0})  \big\|^2 \Big|   + 
\frac{\beta_{\textbf{I}}}{ \| x_{2,0} \|^2}   \Big| e^{\alpha t} \langle \ \nabla^2 f(x_1) x_2 ,x_2 \rangle  
- \langle \nabla^2 f(x_{1,0}) x_{2,0} , x_{2,0} \rangle \Big| ,
\end{aligned}
\end{equation*}
where we made use of the triangle inequality in the inequality (i) and the relation~(\ref{z5}) in the equality (ii), respectively. Based on the above discussion, we then conclude that
\begin{equation*}
\label{z9}
\begin{aligned}
& \big|u_{\textbf{I}}( x ) - u_{\textbf{I}}( x_{,0} )\big| \\
&  \overset{\text{(i)}}{\leq} \frac{\beta_{\textbf{I}}}{ \| x_{2,0} \|^2}  
\left( \frac{L_f}{\delta \beta_{\textbf{I}}}  e^{\alpha t} \big(\beta_{\textbf{I}} C e^{\delta t} + 1 \big) \big(e^{\delta t} - 1 \big)  +
\frac{1}{\beta^2_{\textbf{I}}} \big(e^{\alpha t} - 1 \big) \right)  \|  x_{2,0} \|^2\\
&\quad + \frac{\beta_{\textbf{I}}}{ \| x_{2,0} \|^2} \left(
\frac{ H_f }{\delta} e^{(\alpha+2\delta) t}  \|  x_{2,0} \| \cdot \big(e^{\delta t} - 1 \big)  + \mathcal{L}_f  \big( e^{(\alpha+\delta) t} + e^{\alpha t} \big) \big(e^{\delta t} - 1 \big)
+ \mathcal{L}_f \big(e^{\alpha t} -1 \big) \right)  \| x_{2,0}  \|^2 \\
&  \overset{\text{(ii)}}{\leq} 
 \frac{L_f}{\delta}  e^{\delta t} (\beta_{\textbf{I}} C e^{\delta t} + 1 ) (e^{\delta t} - 1 ) +
\frac{1}{\beta_{\textbf{I}}} (e^{\delta t} - 1) 
\\
&\quad + \beta_{\textbf{I}} \Big(
 \beta_{\textbf{I}} H_f \delta^{-1} \cdot e^{3 \delta t}  \big\|\nabla f(x_{1,0})\big\| \cdot \big(e^{\delta t} - 1 \big)      + \mathcal{L}_f  \big(e^{2 \delta t} + e^{\delta t} \big) \big(e^{\delta t} - 1 \big)  + \mathcal{L}_f \big(e^{\delta t} -1 \big) \Big) \\
&  = \Big(  L_f \delta^{-1} \cdot  e^{\delta t} (\beta_{\textbf{I}} C e^{\delta t} + 1 ) + \frac{1}{\beta_{\textbf{I}}} +
\frac{\beta_{\textbf{I}}^2 H_f }{\delta} e^{3 \delta t}  \big\|\nabla f(x_{1,0})\big\|    + \beta_{\textbf{I}} \mathcal{L}_f  (e^{2 \delta t} + e^{\delta t})
+ \beta_{\textbf{I}} \mathcal{L}_f   \Big) \big( e^{\delta t} - 1 \big)\\
& =: \phi\Big(t,\big\|\nabla f(x_{1,0})\big\|\Big),
\end{aligned}
\end{equation*}
where the inequality (i) follows from the implications of Steps 2 and 3, and the equality (ii) is an immediate consequence of the relation $\alpha < \delta$ and the equality $x_{2,0}=-\beta_{\textbf{I}} \nabla f( x_{1,0})$.

\textbf{Step 5:} Consider $a_1$ defined in (\ref{z_t2_00}) 
and recall that $u_{\textbf{I}} ( x_{,0})$ by design lies inside the input interval $[\ul{u}_{\textbf{I}} , \ol{u}_{\textbf{I}}]$. The quantity $a_1$ is a lower bound on the distance of $u_{\textbf{I}}( x_{,0})$ to the boundaries of the interval $[\ul{u}_{\textbf{I}} , \ol{u}_{\textbf{I}}]$. Thus, the inter-jump interval $\tau_{\textbf{I}}$ satisfies
\begin{align*}
\tau_{\textbf{I}}  \geq \max \left\{t\geq 0 :~ \big|u_{\textbf{I}}( x ) - u_{\textbf{I}}( x_{,0} )\big| \leq a_1  \right\}
 \geq \max \left\{t\geq 0 :~ \phi\Big(t,\big\|\nabla f(x_{1,0})\big\|\Big) \leq a_1  \right\},
\end{align*} 
where the second inequality is implied by the analysis provided in Step 4. Consider a positive constant $r>1$. One can infer for every $t\in \big[0, \delta^{-1}{\log r}\big]$ that
\begin{align*}
 \phi\Big(t,\big\|\nabla f(x_{1,0})\big\|\Big) 
& \leq \Big(  r L_f \delta^{-1}   (r \beta_{\textbf{I}} C  + 1 ) + \beta_{\textbf{I}}^{-1}   
+ r^3 \beta_{\textbf{I}}^2 H_f \delta^{-1}  \big\|\nabla f(x_{1,0})\big\| \\
  &\quad
  + (r^2+r) \beta_{\textbf{I}} \mathcal{L}_f + \beta_{\textbf{I}} \mathcal{L}_f   \Big) (e^{\delta t} - 1)\\
&= \Big(a_2 + a_3 \big\|\nabla f(x_{1,0})\big\| \Big)(e^{\delta t} - 1) \\
& =: \phi'\Big(t,\big\|\nabla f(x_{1,0})\big\|\Big),
\end{align*}
where the constants $a_2$ and $a_3$ are defined in \eqref{z_t2_01}, \eqref{z_t2_02}, respectively, and the inequality $e^{\delta t} < r$ is used. Suppose now $\tau'$ is the lower bound of the inter jump in \eqref{z_t2_03}. Then $\phi'\Big(\tau',\big\|\nabla f(x_{1,0})\big\|\Big)=a_1$,  where the constant $a_1$ is defined in \eqref{z_t2_00}. It is straightforward to establish the assertion made in \eqref{z_t2_03}.

In the second part of the assertion, we should show that the proposed lower bound in \eqref{z_t2_03} is uniformly away from zero along any trajectories of the hybrid system. To this end, we only need to focus on the term $\|\nabla f\big(x_{1}(t)\big)\|$. Recall that Theorem~\ref{theo_1b} effectively implies that $\lim_{t\rightarrow \infty}~\|\nabla f\big(x_{1}(t)\big)\| = 0$, possibly not in a monotone manner though. This observation allows us to deduce that $M:= \sup_{t \ge 0}\|\nabla f\big(x_{1}(t)\big)\| < \infty$. Using the uniform bound $M$, we have a minimum non-zero inter-jump interval, giving rise to a Zeno-free behavior for all solution trajectories. 

\subsection{Proof of Theorem~\ref{theo_step_conv}} 

The proof follows a similar idea as in \cite[Theorem~3.1]{armanICML} but the required technical steps are somewhat different, leading to another set of technical assumptions. In the first step, we begin with describing on how the chosen input $u_{\textbf{II}}(x)$ in \eqref{step_02} ensures achieving the desired exponential convergence rate $\mathcal{O}\big(e^{-\alpha t}\big)$. 
Let us define the set $\mathcal{E}_{\alpha}:=\Big\{x\in \mathbb{R}^{2n}: \alpha \big(f(x_1)-f^*\big) < \langle \nabla f(x_1), -x_2  \rangle \Big\}$. 
We demonstrate that as long as a solution trajectory of the continuous flow~\eqref{step_01} is contained in the set~$\mathcal{E}_{\alpha}$, the function $f$  obeys the exponential decay~\eqref{eqt_8b}. To this end, observe that if $\big(x_1(t),x_2(t)\big) \in \mathcal{E}_{\alpha}$,
\begin{align*}
	\frac{d}{dt}\Big(f\big(x_1(t)\big)-f^*\Big) = \big\langle \nabla f\big(x_1(t)\big),x_2(t) \big\rangle  \le -\alpha \big(f(x_1)-f^*\big).
\end{align*}
The direct application of Gronwall's inequality, see \cite[Lemma~A.1]{khalil1996noninear}, to the above inequality yields the desired convergence claim~\eqref{eqt_8b}. 
Hence, it remains to guarantee that the solution trajectory renders the set $\mathcal{E}_{\alpha}$ invariant.
Let us define the quantity
\begin{align*}
	\sigma(t) := \langle \nabla f\big(x_1(t)\big),x_2(t) \rangle +\alpha\Big(f\big(x_1(t)\big)-f^*\Big).
\end{align*}
By construction, if $\sigma(t) < 0$, it follows that $\big(x_1(t),x_2(t)\big) \in \mathcal{E}_{\alpha}$. As a result, if we synthesize the feedback input $u_{\textbf{II}}(x)$ such that $\dot\sigma(t) \le 0$ along the solution trajectory of \eqref{step_01}, the value of $\sigma(t)$ does not increase, and as such 
\begin{align*}
	\big(x_1(t),x_2(t)\big) \in \mathcal{E}_{\alpha}, ~\forall t \ge 0 ~ \Longleftrightarrow ~ \big(x_1(0),x_2(0)\big) \in \mathcal{E}_{\alpha}.
\end{align*}
To ensure non-positivity property of $\dot{\sigma}(t)$, note that we have
\begin{align*}
\dot{\sigma}(x)
&=  \langle \nabla^2 f(x_1) x_2, x_2  \rangle + \langle \nabla f(x_1), \dot{x}_2 \rangle +\alpha \langle  \nabla f(x_1), x_2 \rangle \\
 & =  \langle \nabla^2 f(x_1) x_2, x_2  \rangle + \langle \nabla f(x_1), - x_2 - u_{\textbf{II}} ( x ) \nabla f(x_1) \rangle  + \alpha \langle  \nabla f(x_1), x_2 \rangle \\
 &  =  \langle \nabla^2 f(x_1) x_2, x_2  \rangle + \langle \nabla f(x_1), - x_2 \rangle - u_{\textbf{II}} ( x ) \| \nabla f(x_1) \|^2  - \alpha \langle  \nabla f(x_1), - x_2 \rangle \\
 &    =  \langle \nabla^2 f(x_1) x_2, x_2  \rangle + ( 1- \alpha) \langle \nabla f(x_1), - x_2 \rangle  - u_{\textbf{II}} ( x ) \| \nabla f(x_1) \|^2 =  0,
\end{align*}
where the last equality follows from the definition of the proposed control law~\eqref{step_02}. It is worth noting that one can simply replace the information of the Hessian $\nabla^2 f\big(x_1(t)\big)$ with the upper bound $L_f$ and still arrive at the desired inequality, see also Remark~\ref{rem_40} with regards to the 1st-order information oracle. 
Up to now, we showed that the structure of the control feedback guarantees the $\alpha$-exponential convergence. 
It remains then to ensure that $x(0) \in \mathcal{E}_{\alpha}$.
Consider the initial state $x_2(0) =-\beta_{\text{II}} \nabla f\big(x_1(0)\big)$. Notice that
\begin{align*}
 \alpha \Big(f\big(x_1(0)\big)-f^*\Big) & \leq \frac{\alpha}{2 \mu_f} \big\| \nabla f\big(x_1(0)\big) \big\|^2 
  = \frac{\alpha}{2 \mu_f \beta_{\textbf{II}}} \langle -x_2(0), \nabla f\big(x_1(0)\big)  \rangle  
  \leq \langle \nabla f\big(x_1(0)\big) , -x_2(0)   \rangle,
\end{align*}
where in the first inequality we use the gradient-dominated assumption~\eqref{d_1}, and in the second inequality the condition~\eqref{step_05_3} is employed. 
Suppose $\big(x_1^{\top}(0),x^{\top}_2(0)\big)^{\top}$ as the jump state $x^+$. It is evident that the range space of the jump map~\eqref{step_04} lies inside the set~$\mathcal{E}_\alpha$. 
At last, it is required to show that the jump policy is well-defined in the sense that the trajectory lands in the interior of the flow set $\mathcal{C}_{\textbf{I}}$~\eqref{step_03}, i.e., the control values also belong to the admissible set $[\ul{u}_{\textbf{II}},\ol{u}_{\textbf{II}}]$.  
To this end, we only need to take into account the initial control value since the switching law is continuous in the states and serves the purpose by design. 
Suppose that $x^+ \in \mathcal{C}_{\textbf{II}}$, we then have the sufficient requirements
\begin{multline*}
\ul{u}_{\textbf{II}} < \frac{-\ell_f\beta_{\text{II}}^2 \|\nabla f(x_1^+) \|^2+ (1-\alpha) \beta_{\text{II}} \| \nabla f(x_1^+) \|^2 }{ \| \nabla f(x_1^+) \|^2 }
\\ \le u_{\textbf{II}}(x^+) \leq \\
\frac{L_f\beta_{\text{II}}^2 \|\nabla f(x_1^+) \|^2+ (1-\alpha) \beta_{\text{II}} \| \nabla f(x_1^+) \|^2 }{\| \nabla f(x_1^+) \|^2 } < \ol{u}_{\textbf{II}},
\end{multline*}
where the relations \eqref{step_02} and \eqref{p2} are considered. Factoring out the term $\| \nabla f(x_1^+) \|^2$ leads to the sufficiency requirements given in \eqref{step_05_2} and \eqref{step_05_1}. Hence, the claim of Theorem~\ref{theo_step_conv} follows.

\subsection{Proof of Theorem~\ref{theo_step_zeno}} 
In order to facilitate the argument regarding the proof of Theorem~\ref{theo_step_zeno}, we begin with providing a lemma describing the norm-2 behaviors of $\langle \nabla f(x_1), - x_2 \rangle$, $x_2$, and $\nabla f(x_1)$. For the sake of brevity, we employ the same notations used in Subsection~\ref{subsec:pro_zeno1}, as well.

\begin{Lem}[Growth bounds]
\label{lem_step1}
Consider the continuous-time hybrid control system (\ref{p1}) with the respective parameters (\ref{sH_step}) satisfying (\ref{step_05}) where the function $f$ satisfies Assumptions~(\ref{p2}) and (\ref{d_1}). Suppose the hybrid control system is initiated from $\big(x_{1,0},\beta_{\textbf{II}} \nabla f(x_{1,0})\big)$ for some $x_{1,0}\in \mathbb{R}^n$. Then,
\begin{subequations}
\label{step_1}
\begin{align}
& \langle \nabla f(x_1), - x_2 \rangle =\beta_{\textbf{II}} e^{-\alpha t} \| \nabla f(x_{1,0})\|^2,  \label{step_1_0} \\
& \| x_2 \| \leq D(t) \| \nabla f ( x_{1,0}) \|, \label{step_1_1}\\
& \ul{\eta}(t) \| \nabla f ( x_{1,0}) \| \leq \| \nabla f ( x_{1}) \| \leq \ol{\eta}(t) \| \nabla f ( x_{1,0}) \|, \label{step_1_2}
\end{align}
\end{subequations}
with the time-varying scalars $D$, $\ul{\eta}$, and $\ol{\eta}$ given by
\begin{subequations}
\label{step_2}
\begin{align}
& D(t):= \Big( \beta_{\textbf{II}} ^2 e^{-2t} +\beta_{\textbf{II}} U \big( 1 - e^{-2t}  \big)  \Big)^{\frac{1}{2}}, \label{step_2_1}\\
& \ul{\eta}(t) :=  1 - \mathcal{L}_f (\beta_{\textbf{II}}^2+\beta_{\textbf{II}} U)^{\frac{1}{2}} t , \label{step_2_2}\\
& \ol{\eta}(t) :=  1 + \mathcal{L}_f (\beta_{\textbf{II}}^2+\beta_{\textbf{II}} U)^{\frac{1}{2}} t , \label{step_2_3}
\end{align}
\end{subequations}
respectively, where $U:=\max \{\ol{u}_{\textbf{II}}, -\ul{u}_{\textbf{II}}\}$ and $\mathcal{L}_f:= \max \{ \ell_f, L_f \}$.
\end{Lem}
\begin{proof}
Considering the flow dynamics~\eqref{step_01} and the feedback input~\eqref{step_02}, one obtains
\begin{align*}
 \frac{d}{dt}\langle \nabla f(x_1), - x_2 \rangle 
&= \langle \nabla^2 f(x_1)x_2, -x_2 \rangle + \langle \nabla f(x_1),-\dot{x}_2 \rangle \\
& = \langle \nabla^2 f(x_1)x_2, -x_2 \rangle + \langle \nabla f(x_1),x_2 +u_{\textbf{II}}(x) \nabla f(x_1) \rangle \\
& = \langle \nabla^2 f(x_1)x_2, -x_2 \rangle + \langle \nabla f(x_1),x_2 \rangle +u_{\textbf{II}}(x) \|\nabla f(x_1) \|^2 \\
& = \langle \nabla^2 f(x_1)x_2, -x_2 \rangle + \langle \nabla f(x_1),x_2 \rangle  +\langle \nabla^2 f(x_1)x_2, x_2 \rangle - (1-\alpha) \langle \nabla f(x_1),x_2 \rangle \\
&=-\alpha \langle \nabla f(x_1), - x_2 \rangle,
\end{align*}
and as a result given the initial state $\big(x_{1,0}, -\beta_{\textbf{II}} \nabla f(x_{1,0})\big)$, the equality given in \eqref{step_1_0} is valid. We next turn to establish that \eqref{step_1_1} holds. Let us define $h(t)=\| x_2 \|^2$. Hence,
\begin{align*}
\frac{d}{dt}h(t) & \overset{\text{(i)}}{=} 2 \langle x_2, -x_2 -u_{\textbf{II}} (x)\nabla f(x_1) \rangle = -2\|x_2 \|^2 + 2 u_{\textbf{II}} (x) \langle \nabla f(x_1),-x_2  \rangle\\
&\overset{\text{(ii)}}{=} -2 h(t) + 2 u_{\textbf{II}} (x) \beta_{\textbf{II}} e^{-\alpha t} \| \nabla f(x_{1,0}) \|^2 \leq  -2 h(t) + 2 U \beta_{\textbf{II}} \| \nabla f(x_{1,0}) \|^2,
\end{align*}
where we made use of the flow dynamics~\eqref{step_01} in the inequality (i) and the equation~\eqref{step_1_0} in the equality (ii). We then use the Gronwall's inequality to infer that 
\begin{align*}
\|x_2\|^2 
& \leq e^{-2t} \| x_{2,0} \|^2 + \int_0^t e^{-2(t-\tau)}2U\beta_{\textbf{II}} \big\| \nabla f(x_{1,0}) \big\|^2 d \tau\\
& = e^{-2t} \beta_{\textbf{II}}^2 \big\| \nabla f(x_{1,0}) \big\|^2 +e^{-2t}2U\beta_{\textbf{II}} \big\| \nabla f(x_{1,0}) \big\|^2 \int_0^t e^{2\tau} d \tau\\
& = e^{-2t}  \big\| \nabla f(x_{1,0}) \big\|^2 \Big( \beta_{\textbf{II}}^2 e^{-2t} +\beta_{\textbf{II}} U \big( 1 - e^{-2t} \big) \Big)\\
& =: D^2(t) \big\| \nabla f(x_{1,0}) \big\|^2,
\end{align*}
where $D(t)$ is defined in \eqref{step_2_1}. As a result, the claim in \eqref{step_1_1} holds. The argument to show the last claim in Lemma~\ref{lem_step1} is discussed now. Let us define $g(t):=\big\| \nabla f(x_1) \big\|^2$. Observe that
\begin{align*}
\frac{d}{dt} g(t) = 2 \langle  \nabla^2 f(x_1) x_2 , \nabla f(x_1) \rangle,
\end{align*}
and as a result 
\begin{align*}
\left| \frac{d}{dt} g(t) \right| 
\overset{\text{(i)}}{\leq} 2 \mathcal{L}_f \| x_2 \| \cdot \big\| \nabla f(x_1) \big\| 
= 2 \mathcal{L}_f \| x_2 \|  \sqrt{g(t)} 
\overset{\text{(ii)}}{\leq} 
 2 \mathcal{L}_f D(t) \big\| \nabla f(x_{1,0}) \big\| \sqrt{g(t)} ,
\end{align*}
where the inequalities (i) and (ii) are implied by Assumption~\eqref{p2} and the inequality~\eqref{step_1_1}, respectively. Hence, we deduce that
\begin{align*}
\frac{d}{dt} g(t) \geq - 2 \mathcal{L}_f D(t) \big\| \nabla f(x_{1,0}) \big\| \sqrt{g(t)} ,
\end{align*}
and as a consequence
\begin{align*}
\frac{d  g(t)}{\sqrt{g(t)}} \geq - 2 \mathcal{L}_f D(t) \big\| \nabla f(x_{1,0}) \big\| dt.
\end{align*}
Integrating the two sides of the above inequality results in   
\begin{align*}
\sqrt{g(t)}  -\sqrt{g(0)} 
& \geq  - \mathcal{L}_f  \big\| \nabla f(x_{1,0}) \big\| \int_0^t D(\tau) d\tau\\
& = - \mathcal{L}_f  \big\| \nabla f(x_{1,0}) \big\| \int_0^t \Big( \beta_{\textbf{II}} ^2 e^{-2\tau} +\beta_{\textbf{II}} U \big( 1 - e^{-2\tau}  \big)  \Big)^{\frac{1}{2}} d\tau\\
& \geq - \mathcal{L}_f  \big\| \nabla f(x_{1,0}) \big\| \int_0^t \big( \beta_{\textbf{II}} ^2 +\beta_{\textbf{II}} U   \big)^{\frac{1}{2}} d\tau\\
& = - \mathcal{L}_f  \big\| \nabla f(x_{1,0}) \big\| \big( \beta_{\textbf{II}} ^2 +\beta_{\textbf{II}} U   \big)^{\frac{1}{2}} t.
\end{align*}
Based on the above analysis and the definition of $g(t)$, it follows that 
\begin{align*}
\big\| \nabla f(x_1) \big\| \geq \ul{\eta}(t) \big\| \nabla f(x_{1,0}) \big\|,
\end{align*}
where $\ul{\eta}(t)$ is given in \eqref{step_2_2}. Proceeding with a similar approach to the one presented above, one can use the inequality 
\begin{align*}
\frac{d}{dt} g(t) \leq 2 \mathcal{L}_f D(t) \big\| \nabla f(x_{1,0}) \big\| \sqrt{g(t)} ,
\end{align*}
and infer that
\begin{align*}
\big\| \nabla f(x_1) \big\| \leq \ol{\eta}(t) \big\| \nabla f(x_{1,0}) \big\|,
\end{align*}
where $\ol{\eta}(t)$ is defined in \eqref{step_2_3}. Thus, the last claim in Lemma~\ref{lem_step1} also holds.
\end{proof}

\textbf{Proof of Theorem~\ref{theo_step_zeno}:}
We are now in a position to formally state the proof of Theorem~\ref{theo_step_zeno}. Consider the parameter $\delta$ as defined in Theorem~\ref{theo_step_zeno}. Intuitively, this quantity represents a lower bound on the distance of $u_{\textbf{II}} (0)$ from the endpoints of the flow set interval. Thus, one can obtain a lower bound on the inter-jump interval $\tau_{\textbf{II}}$ as follows
\begin{align}
\label{step_4_0}
\tau_{\textbf{II}} \geq \sup~\{t>0:~|u_{\textbf{II}}(t)-u_{\textbf{II}}(0)|\leq \delta \}.
\end{align}
On the other hand, given the structure of $u_{\textbf{II}}$ in \eqref{step_02},
\begin{align*}
-\frac{\ell_f \| x_2 \|^2}{\| \nabla f(x_1) \|^2} + (1-\alpha) \frac{\beta_{\textbf{II}} e^{-\alpha t} \| \nabla f(x_{1,0}) \|^2}{\| \nabla f(x_1) \|^2} 
\leq u_{\textbf{II}}(t) \leq 
\frac{L_f \| x_2 \|^2}{\| \nabla f(x_1) \|^2} + (1-\alpha) \frac{\beta_{\textbf{II}} e^{-\alpha t} \| \nabla f(x_{1,0}) \|^2}{\| \nabla f(x_1) \|^2},
\end{align*}
since the function $f$ satisfies Assumption \eqref{p2}. 
In light of Lemma~\ref{lem_step1} and considering the above relation, one can infer that for $\alpha\leq 1$, we name Case(i),
\begin{subequations}
\label{step_5}
\begin{align}
\ul{e}(t):=-\frac{\ell_f D(t)^2}{\ul{\eta}(t)^2} + (1-\alpha) \frac{\beta_{\textbf{II}} e^{-\alpha t} }{\ol{\eta}(t)^2} 
\leq u_{\textbf{II}}(t) \leq 
\frac{L_f D(t)^2}{\ul{\eta}(t)^2} + (1-\alpha) \frac{\beta_{\textbf{II}} e^{-\alpha t} }{\ul{\eta}(t)^2} =: \ol{e}(t), \label{step_5_1}
\end{align}
and that for $\alpha >1$, we denote by Case (ii),
\begin{align}
\ul{p}(t):=-\frac{\ell_f D(t)^2}{\ul{\eta}(t)^2} + (1-\alpha) \frac{\beta_{\textbf{II}} e^{-\alpha t} }{\ul{\eta}(t)^2} 
\leq u_{\textbf{II}}(t) \leq 
\frac{L_f D(t)^2}{\ul{\eta}(t)^2} + (1-\alpha) \frac{\beta_{\textbf{II}} e^{-\alpha t} }{\ol{\eta}(t)^2} =: \ol{p}(t). \label{step_5_2}
\end{align}
\end{subequations}
According to the above discussion, we employ \eqref{step_5} to obtain a lower bound $\tau_{\textbf{II}}$ instead of using \eqref{step_4_0}. Consider a time instant $t_\circ$ such that $t_\circ < 1/ \omega$ where $\omega$ is defined in Theorem~\ref{theo_step_zeno}.

\textbf{Case (i) ($\alpha \leq 1$):} 
Let us denote $\sup_{t\in [0,t_\circ]}\dot{\ol{e}}(t)$ by $b_1$.
Observe that
\begin{align*}
 \dot{\ol{e}}(t)
&= \frac{2 L_f \beta_{\textbf{II}} e^{-2t} (-\beta_{\textbf{II}}+U )(1-\omega t)^2 + 2 \omega (1 - \omega t) L_f \beta_{\textbf{II}} \big( \beta_{\textbf{II}} e^{-2t} +U (1 - e^{-2t})  \big) }{(1-\omega t)^4}\\
& + (1-\alpha) \frac{-\alpha \beta_{\textbf{II}} e^{-\alpha t}(1-\omega t)^2 + 2 \omega (1 - \omega t) \beta_{\textbf{II}} e^{-2t}}{(1-\omega t)^4}\\
& \leq \frac{2 L_f \beta_{\textbf{II}} U e^{-2t} (1-\omega t)^2 + 
2 \omega (1 - \omega t) L_f \beta_{\textbf{II}} \big( \beta_{\textbf{II}} e^{-2t} +U (1 - e^{-2t})  \big)}{(1-\omega t)^4}\\
& + (1-\alpha) \frac{ 2 \omega (1 - \omega t) \beta_{\textbf{II}} e^{-2t}}{(1-\omega t)^4} \\
& \leq \frac{2 L_f \beta_{\textbf{II}} \big( U + \omega ( \beta_{\textbf{II}}  +U)   \big)}{(1-\omega t)^3}  + (1-\alpha) \frac{ 2 \omega \beta_{\textbf{II}} }{(1-\omega t)^3} \\
& \leq \frac{2 L_f \beta_{\textbf{II}} \big( U + \omega ( \beta_{\textbf{II}}  +U)   \big)}{(1-\omega t_\circ)^3}  + (1-\alpha) \frac{ 2 \omega \beta_{\textbf{II}} }{(1-\omega t_\circ)^3} =: b_1,
\end{align*}
considering \eqref{step_5_1}. Hence, $\ol{e}(t) \leq b_1 t + \ol{e}(0)$ and as a result
\begin{align}
\label{step_6}
\tau_{\textbf{II}} &\geq \max \{t\in (0,t_\circ]:~ b_1 t +\ol{e}(0) - \ol{e}(0) \leq \delta  \}
= \min\{ t_\circ, \delta/b_1 \},
\end{align}
by virtue of the fact that $b_1 t +\ol{e}(0)$ is a monotonically increasing function that upper bounds $u_{\textbf{II}}(t)$. 
Now, let us define $b_2:= \inf_{t\in (0,t_\circ]}\dot{\ul{e}}(t)$.
Notice that 
\begin{align*}
\dot{\ul{e}}(t) 
& = \frac{2 \ell_f \beta_{\textbf{II}} e^{-2t} (\beta_{\textbf{II}}-U )(1-\omega t)^2 -
2 \omega (1 - \omega t) \ell_f \beta_{\textbf{II}} \big( \beta_{\textbf{II}} e^{-2t} +U (1 - e^{-2t})  \big) }{(1-\omega t)^4} \\
& + (1-\alpha) \frac{-\alpha \beta_{\textbf{II}} e^{-\alpha t}(1+\omega t)^2 - 2 \omega (1 + \omega t) \beta_{\textbf{II}} e^{-2t}}{(1+\omega t)^4}\\
& \geq \frac{-2 \ell_f \beta_{\textbf{II}} e^{-2t} U (1-\omega t)^2  -
2 \omega (1 - \omega t) \ell_f \beta_{\textbf{II}} \big( \beta_{\textbf{II}} e^{-2t} +U (1 - e^{-2t})  \big) }{(1-\omega t)^4} \\
& - (1-\alpha) \frac{\alpha \beta_{\textbf{II}} e^{-\alpha t}(1+\omega t)^2 + 2 \omega (1 + \omega t) \beta_{\textbf{II}} e^{-2t}}{(1+\omega t)^4}\\
& \geq -\frac{2 \ell_f \beta_{\textbf{II}}  \big(U  + \omega ( \beta_{\textbf{II}} +U ) \big) }{(1-\omega t_\circ)^3}  - (1-\alpha) \frac{\alpha \beta_{\textbf{II}} (1+\omega t_\circ) + 2 \omega  \beta_{\textbf{II}} }{1} =: -b_2.
\end{align*}
 Thus, $\ul{e}(t)\geq -b_2 t + \ul{e}(0)$ and as a consequence
\begin{align}
\label{step_6}
\tau_{\textbf{II}} & \geq \max \{t\in (0,t_\circ]:~ \ul{e}(0) - \big(-b_2t + \ul{e}(0)\big)\leq \delta  \} 
= \min \{ t_\circ, \delta/b_2 \},
\end{align}
because the function $-b_2t + \ul{e}(0)$ is a monotonically decreasing function that lower bounds $u_{\textbf{II}}(t)$.

\textbf{Case (ii) ($\alpha > 1$):} Much of this case follows the same line of reasoning used in Case (i). We thus provide only main mathematical derivations and refer the reader to the previous case for the argumentation. Define $b_3:= \sup_{t\in(0,t_\circ]}\dot{\ol{p}}(t)$. One can deduce from \eqref{step_5_2} that
\begin{align*}
 \dot{\ol{p}}(t) 
& = \frac{2 L_f \beta_{\textbf{II}} e^{-2t} (-\beta_{\textbf{II}}+U )(1-\omega t)^2 +
2 \omega (1 - \omega t) L_f \beta_{\textbf{II}} \big( \beta_{\textbf{II}} e^{-2t} +U (1 - e^{-2t})  \big) }{(1-\omega t)^4}\\
& + (1-\alpha) \frac{-\alpha \beta_{\textbf{II}} e^{-\alpha t}(1+\omega t)^2 - 2 \omega (1 + \omega t) \beta_{\textbf{II}} e^{-2t}}{(1+\omega t)^4}\\
& \leq \frac{2 L_f \beta_{\textbf{II}} \big(  U +  \omega  (\beta_{\textbf{II}}  +U)   \big) }{(1-\omega t_\circ)^3}  + (\alpha - 1) \frac{\alpha \beta_{\textbf{II}} (1+\omega t_\circ) + 2 \omega  \beta_{\textbf{II}} }{1} =: b_3.
\end{align*}
Hence, $\ol{p}(t)\leq b_4 t + \ol{p}(0)$ and as a result
\begin{align}
\label{step_8}
\tau \geq \min \{ t_\circ, \delta/b_3 \}.
\end{align}
Finally, define $\dot{\ul{p}}(t):=\inf_{t\in(0,t_\circ]}\ul{p}(t)$ from which it follows that
\begin{align*}
\dot{\ul{p}}(t) 
& = \frac{2 \ell_f \beta_{\textbf{II}} e^{-2t} (\beta_{\textbf{II}}-U )(1-\omega t)^2 - 2 \omega (1 - \omega t) \ell_f \beta_{\textbf{II}} \big( \beta_{\textbf{II}} e^{-2t} +U (1 - e^{-2t})  \big) }{(1-\omega t)^4} \\
& + (1-\alpha) \frac{-\alpha \beta_{\textbf{II}} e^{-\alpha t}(1-\omega t)^2 + 2 \omega (1 - \omega t) \beta_{\textbf{II}} e^{-2t}}{(1-\omega t)^4}\\
& \geq -\frac{2 \ell_f \beta_{\textbf{II}} \big( U  + \omega  (\beta_{\textbf{II}} +U)  \big) }{(1-\omega t_\circ)^3}  - (\alpha-1) \frac{ 2 \omega  \beta_{\textbf{II}} }{(1-\omega t_\circ)^3} =: -b_4,
\end{align*}
considering \eqref{step_5_2}. Now, since $\ul{p}(t) \geq -b_4 t + \ul{p}(0)$, it is implied that 
\begin{align}
\label{step_9}
\tau_{\textbf{II}} \geq \min\{ t_\circ, \delta/b_4 \}.
\end{align}
Notice that based on the relations derived in \eqref{step_6}-\eqref{step_9},
\begin{align*}
\tau_{\textbf{II}} \geq 
 \min\Big\{t_\circ, \frac{2 \mathcal{L}_f \beta_{\textbf{II}} \big(  U +  \omega  (\beta_{\textbf{II}}  +U)   \big) }{(1-\omega t_\circ)^3} + |\alpha - 1| \frac{ 2 \omega  \beta_{\textbf{II}} }{(1-\omega t_\circ)^3} 
 + |\alpha - 1| \alpha \beta_{\textbf{II}} (1+\omega t_\circ)  \Big\}.
\end{align*}
Suppose now for some scalar $r\in (0,1)$, $t_\circ$ is chosen such that $t_\circ \leq \frac{r}{\omega}$. It is evident that
\begin{align*}
\tau_{\textbf{II}} \geq 
 \min\Big\{\frac{r}{\omega}, \delta\Big/\Big(\frac{2 \mathcal{L}_f \beta_{\textbf{II}} \big(  U +  \omega  (\beta_{\textbf{II}}  +U)   \big) }{(1-r)^3} + |\alpha - 1| \frac{ 2 \omega  \beta_{\textbf{II}} }{(1-r)^3} 
 + |\alpha - 1| \alpha \beta_{\textbf{II}} (1+r)\Big)  \Big\}.
\end{align*}
It turns out that the relation \eqref{step_3} in Theorem~\ref{theo_step_zeno} is valid  and this concludes the proof.

\subsection{Proof of Theorem~\ref{theo_2}}
\label{subsec:proof_2}
In what follows, we provide the proof for the structure~\textbf{II} and refer the interested reader to \cite[Theorem~3.7]{armanICML} for the structure~\textbf{I}. We emphasize that the technical steps to establish a stable discretization for both structures are similar.

According to the forward-Euler method, the velocity $\dot{x}_1$ and the acceleration $\dot{x}_2$ in the dynamics~\eqref{p1} with \eqref{sH_step} are discretized as follows:
\begin{subequations}
\begin{align}
\label{dx1}
\frac{x_1^{k+1}- x_1^k} {s} &=  x_2^k,\\
\frac{x_2^{k+1}-x_2^k}{s}   &= - u_{d,\textbf{II}}(x^k) \nabla f(x_1^{k}) - x_2^k,
\end{align}
\end{subequations}
where the discrete input $u_{d,\textbf{II}}(x^k)=u_{\textbf{II}}(x^k)$. Now, observe that the definition of the flow set $\mathcal{C}_{d, \textbf{II}}$ (\ref{d2_1}) implies
\begin{align*}
c_1 \| x^k_2 \|^2 \leq \| \nabla f(x^k_1)  \|^2 \leq c_2 \langle  \nabla f(x^k_1),-x^k_2 \rangle 
 \leq c_2 \| \nabla f(x_1^k) \| \cdot \| x_2^k\|,
\end{align*}
where the extra inequality follows from the Cauchy-Schwarz inequality ($\forall~ a,b\in\mathbb{R}^n$, $\langle a ,b\rangle\leq \| a \|\cdot \| b \|$). In order to guarantee that the flow set $\mathcal{C}_{d,\textbf{II}}$ is non-empty the relation (\ref{d4_d1}) should hold between the parameters $c_1$ and $c_2$ since $\sqrt{c_1}\leq \frac{\| \nabla f(x_1^k) \|}{\| x_2^k \|}\leq c_2$.
Next, suppose that the parameters $c_1$, $c_2$, and $\beta$ satisfy (\ref{d4_d2}). Multiplying (\ref{d4_d2}) by $\|  \nabla f(x_1^k)\|$, one can observe that the range space of the jump map $G_{d,\textbf{II}}(x^k)=\big((x^k)^\top,-\beta \nabla^\top f(x^k)\big)^\top$ is inside the flow set $\mathcal{C}_{d,\textbf{II}}$ (\ref{d2_1}). From the fact that the discrete dynamics (\ref{d1}) evolves respecting the flow set $\mathcal{C}_{d,\textbf{II}}$ defined in (\ref{d2_1}), we deduce
\begin{align*}
  f(x_1^{k+1}) -f(x_1^k)  
 &  \leq  \langle \nabla f(x_1^k), x_1^{k+1} - x_1^k \rangle + \frac{L_f}{2} \| x_1^{k+1} - x_1^k \|^2 \\
 &  \leq  -s \langle \nabla f(x_1^k), -x_2^{k} \rangle + \frac{L_f s^2}{2} \| x_2^k \|^2 \\
 &  <  - \frac{s}{c_2} \| \nabla f(x_1^k) \|^2 + \frac{L_f s^2}{2 c_1} \| \nabla f(x_1^k) \|^2 \\
 &  =  \big( - \frac{s}{c_2} + \frac{L_f}{2 c_1} s^2 \big) \| \nabla f(x_1^k) \|^2   \leq 2 \mu_f \big( - \frac{s}{c_2} + \frac{L_f}{2 c_1} s^2 \big) \big( f(x_1^k)-f^*  \big),
\end{align*}
where we made use of the relation (\ref{p2_g}), the definition (\ref{dx1}), the relation (\ref{d2_1}), and the assumption (\ref{d_1}), respectively. Then, considering the inequality implied by the first and last terms given above and adding $f(x_1^k)-f^*$ to both sides of the considered inequality, we arrive at
\begin{equation*}
f(x_1^{k+1})-f^*\leq \lambda(s,c_1,c_2,\beta)   \left( f(x_1^k)-f^*  \right)
\end{equation*}
where $\lambda(s,c_1,c_2,\beta)$ is given by (\ref{d_r}).  As a result, if the step size $s$ is chosen such that $s<\frac{2c_1}{c_2 L_f}$ then $\lambda(s,c_1,c_2,\beta) \in (0,1)$. The claim of Theorem~\ref{theo_2} follows.
\begin{figure}[t]
	\centering
	\subfigure[Objective value along system trajectories.]{\label{fig:obj_con}\includegraphics[width=70mm]{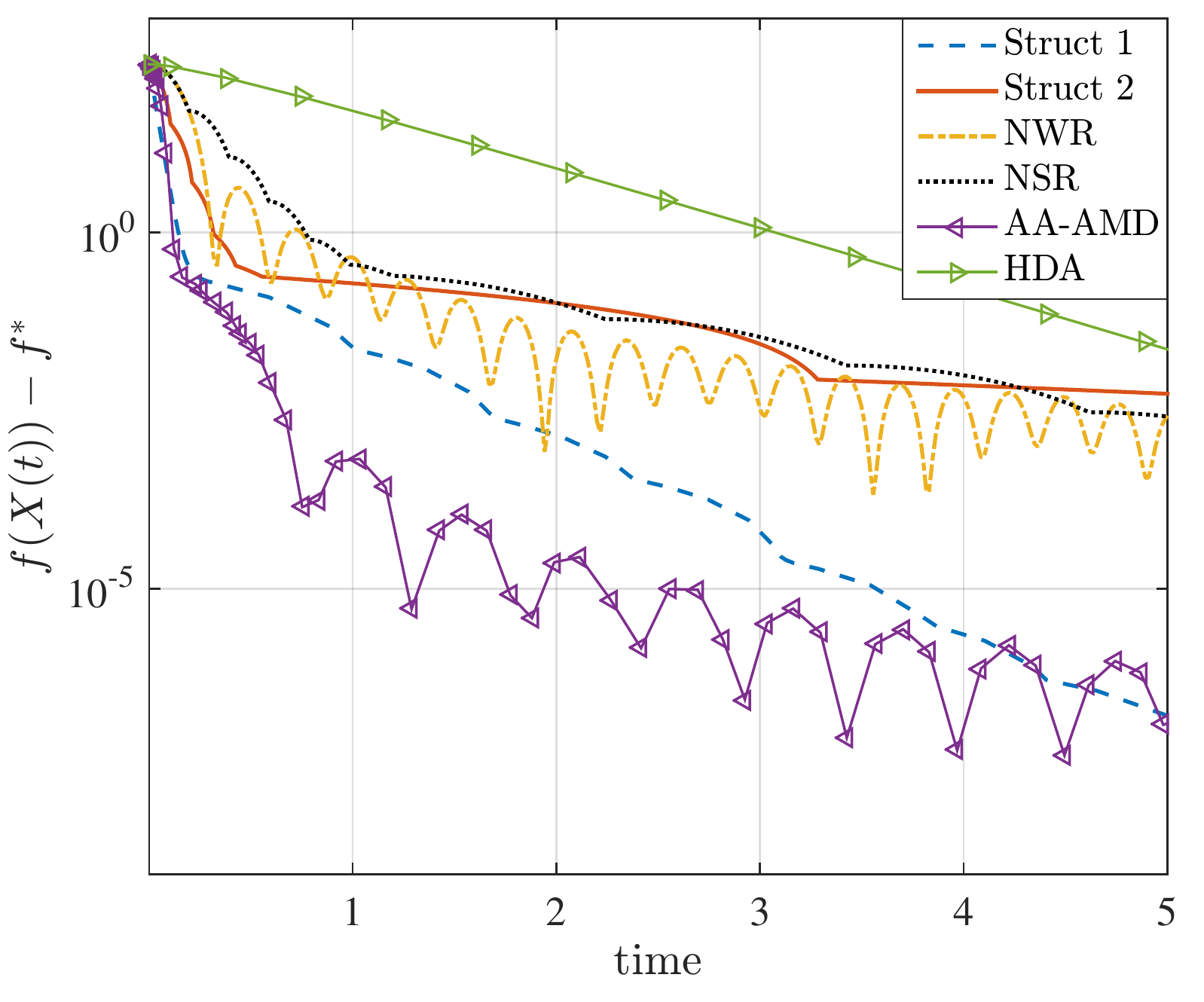}} 
	\qquad   
	\subfigure[State-dependent and time-varying coefficients.]{\label{fig:input_con}\includegraphics[width=70mm]{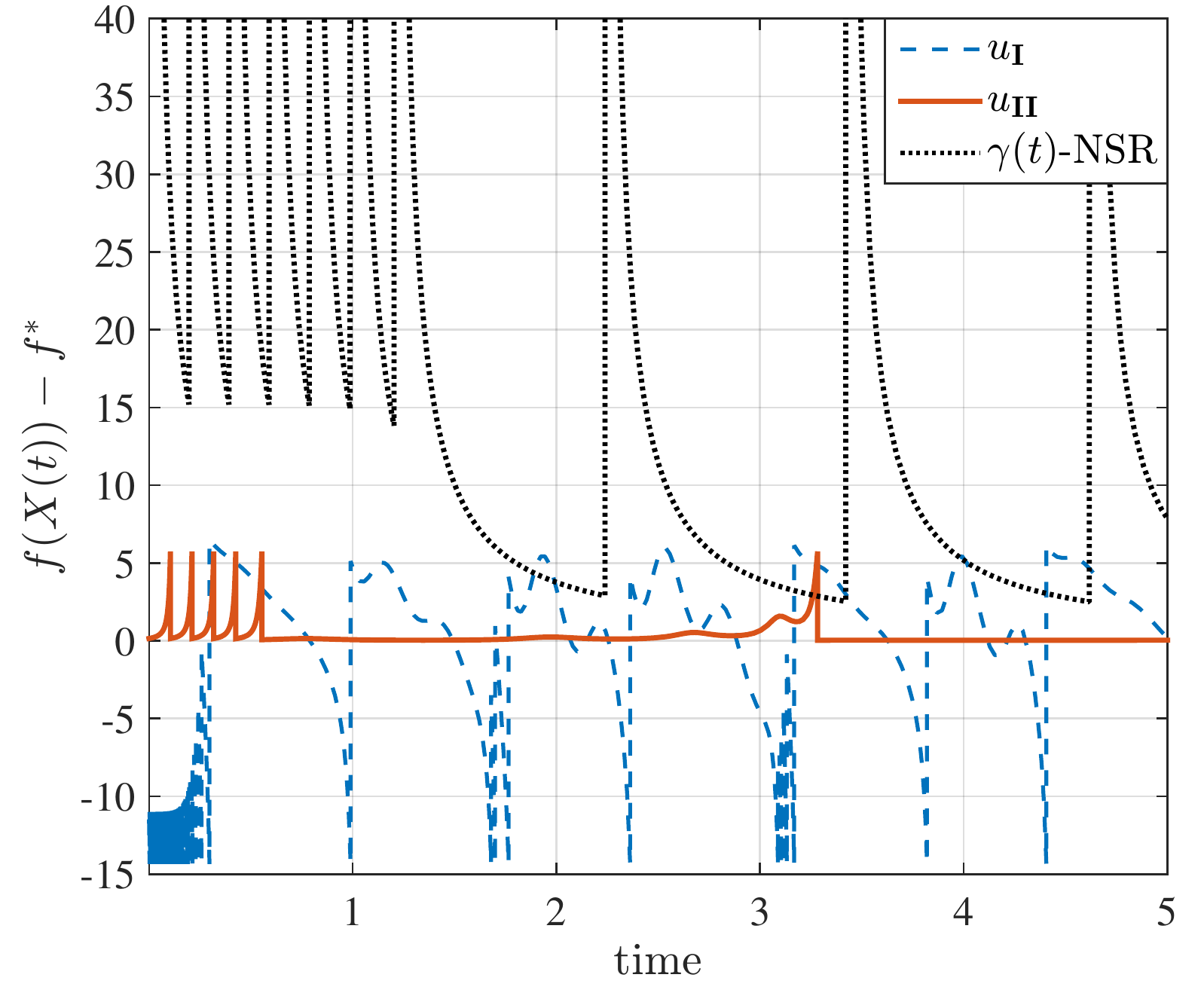}}
	\caption{Continuous-time dynamics of \textbf{Struct~I}, \textbf{Struct~II}, \textbf{NSR}.}
	\label{fig:con_result}
\end{figure}

\section{Numerical Examples}
\label{sec:examp}
In this section a numerical example illustrating the results in this paper is represented. The example is a least mean square error (LMSE) problem $f(X)=\|A X -b \|^2$ where $X\in\mathbb{R}^5$ denotes the decision variable, $A\in\mathbb{R}^{50\times 5}$ with $L_f=2\lambda_{\max}(A^\top A)=136.9832$ and $\mu_f=2\lambda_{\min}(A^\top A)=3.6878$, and $b\in\mathbb{R}^{50}$. Since the LMSE function is convex (in our case, this function is strongly convex), we take $\ell_f=0$. 
We begin with providing the results concerning the continuous-time case. Then, the discrete-time case's results are shown.  

\textbf{Continuous-time case:} In what follows, we compare the behaviors of the proposed structures \textbf{I} and \textbf{II} (denoted by \textbf{Struct I} and \textbf{Struct II}, respectively) 
with the following fast methods:
\begin{itemize}
\item (\textbf{NWR}): Nesterov's fast method~\eqref{dyn2} with $\gamma(t)=\frac{3}{t}$ and without any restarting scheme,
\item (\textbf{NSR}): Nesterov's fast method~\eqref{dyn2} with $\gamma(t)=\frac{3}{t}$  with the speed restarting scheme proposed in \cite[Section~5]{Su2016differential}, 
\item (\textbf{AA-AMD}): the adaptive averaging accelerated mirror descent method proposed in \cite[Section~2]{krichene2016adaptive} with the choice of parameters given in \cite[Example~1]{krichene2016adaptive}, $\beta=3$, and the adaptive heuristic $a(t)=\frac{3}{t}+\text{sign}\big(\max\big\{0,-\langle \nabla f(X(t)),\dot{X}(t)\rangle\big\}\big)\times\frac{1}{t^2}$,
\item (\textbf{HDA}): the Hessian driven accelerated method proposed in \cite{attouch2016fast} with $\alpha=3$ and $\beta=1$.
\end{itemize}
(The notations for some of the parameters involved in the above methods are identical, e.g., the parameter $\beta$ appears in both \textbf{AA-AMD} and \textbf{HDA}. Notice that these parameters are not necessarily the same. We refer the reader to consult with the cited references for more details.) 
We set the desired convergence rates $\alpha_{\textbf{I}}$ and $\alpha_{\textbf{II}}$ equal to each other. We then select  $\beta_{\textbf{I}}$ and $\beta_{\textbf{II}}$ such that the corresponding flow sets $[\ul{u}_{\textbf{I}}, \ol{u}_{\textbf{I}}]$ and $[\ul{u}_{\textbf{II}}, \ol{u}_{\textbf{II}}]$ are relatively close using Theorem~\ref{theo_1b} and Theorem~\ref{theo_step_conv}, respectively. The corresponding parameters of \textbf{Struct I} and \textbf{Struct II} are as follows: $\alpha_{\textbf{I}} = 0.2$, $\beta_{\textbf{I}} =   0.1356$, $ \ul{u}_{\textbf{I}} =   -14.352$, $\ol{u}_{\textbf{I}} =   15.1511$; $\alpha_{\textbf{II}} = 0.2$, $\beta_{\textbf{II}} =   0.0298$, $ \ul{u}_{\textbf{II}} =   -0.1861$, $\ol{u}_{\textbf{II}} =   5.7457$.
 
In Figure~\ref{fig:obj_con}, the behaviors of the suboptimality measure $f\big(X(t)\big)-f^*$ of the considered methods are depicted. The corresponding control inputs of \textbf{Struct I}, \textbf{Struct II}, and \textbf{NSR} are represented in Figure~\ref{fig:input_con}. 
With regards to \textbf{Struct~I}, observe that the length of inter-jump intervals is small during the early stages of simulation. As time progresses and the value of $\nabla f(X)$ decreases, the length of inter-jump intervals relatively increases (echoing the same message conveyed  in Theorem~\ref{theo_zeno}). Furthermore, in the case of \textbf{Struct I} where $u_{\textbf{I}}$ plays the role of damping, the input $u_{\textbf{I}}$ admits a negative range unlike most of the approaches in the literature.

\begin{figure}[t]
	\centering
	\subfigure[Standard tuning parameters.]{\label{fig:obj_disc_all}\includegraphics[width=70mm]{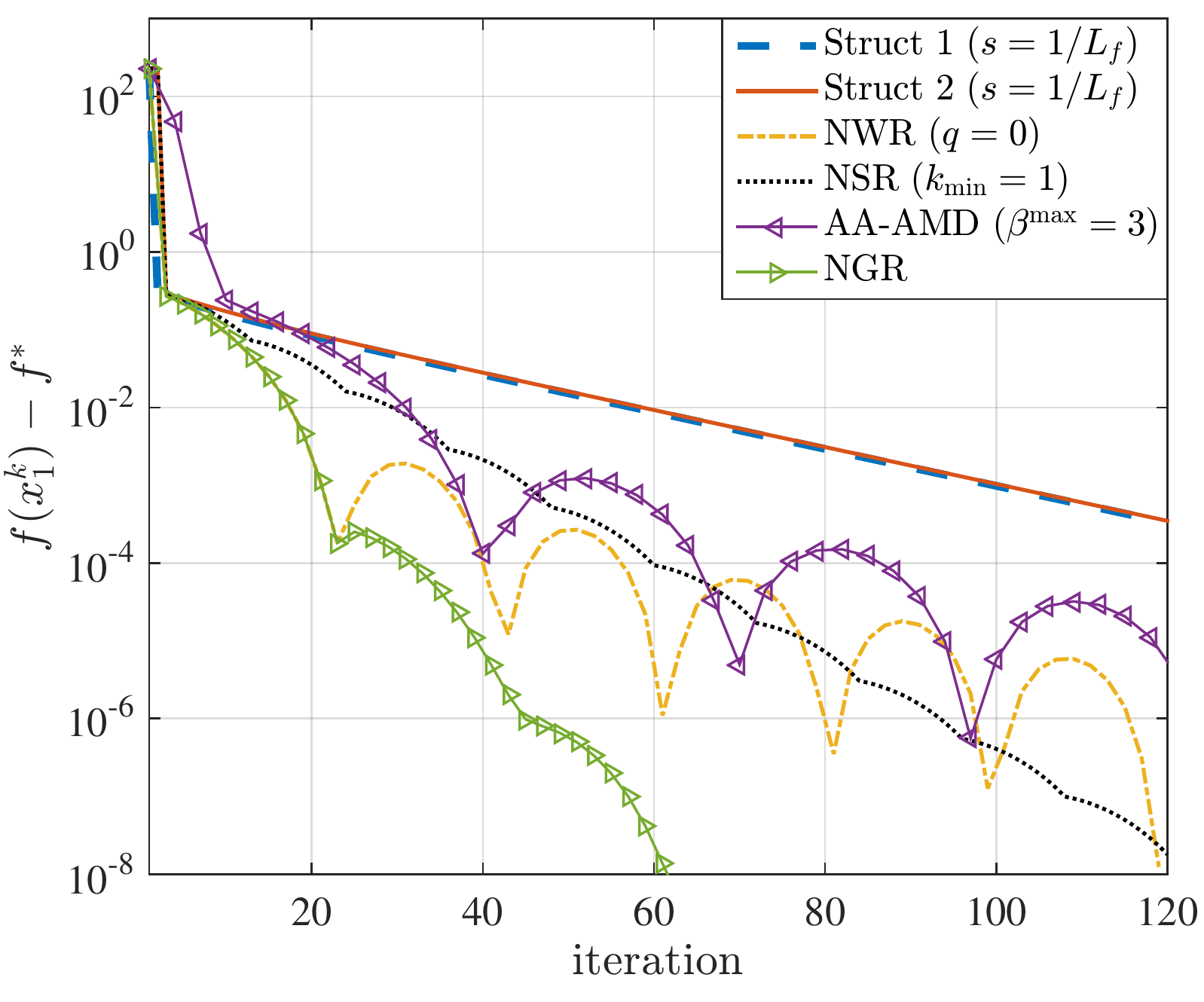}}    
	\qquad
	\subfigure[Example-based optimal tuning parameters.]{\label{fig:obj_disc_bests}\includegraphics[width=70mm]{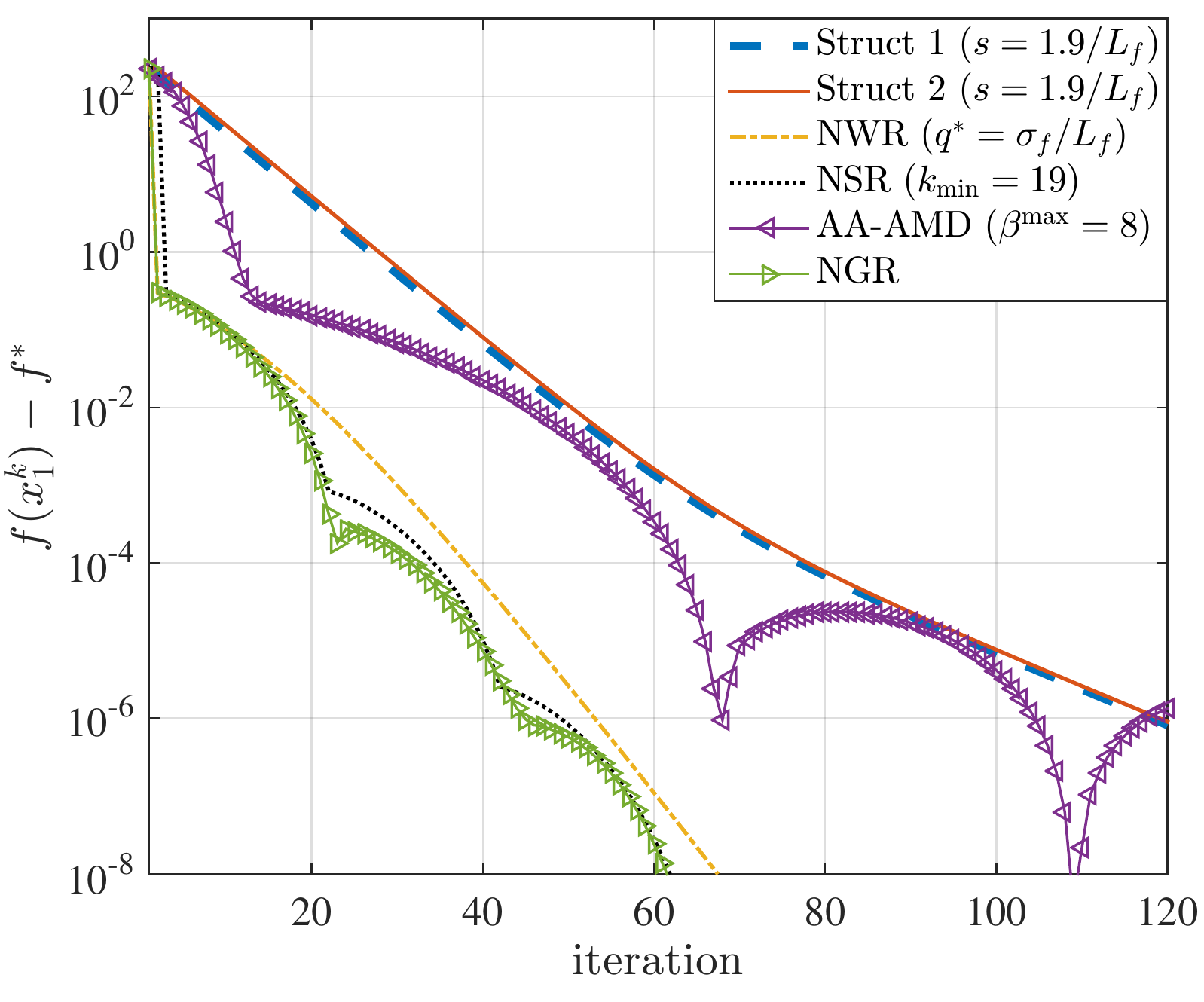}}
	\caption{Discrete-time dynamics of \textbf{Struct~I}, \textbf{Struct~II}, \textbf{NSR}.}
	\label{fig:disc_norm_best}
\end{figure}

\textbf{Discrete-time case:} The discrete-time case's results are now shown. We employ Algorithm~\ref{alg:example} for \textbf{Struct~I} and \textbf{Struct~II}.

In Figure~\ref{fig:obj_disc_all}, we compare these two structures with the discrete-time methods:
\begin{itemize}
\item (\textbf{NWR}): Algorithm~1 in \cite{o2015adaptive} with $q=0$ and $t_k=\frac{1}{L_f}$,
\item (\textbf{NSR}): Algorithm~1 in \cite{Su2016differential} with $k_{\min}=1$ and $s=\frac{1}{L_f}$, 
\item (\textbf{AA-AMD}): Algorithm~1 in the supplementary material of \cite{krichene2016adaptive} with $\beta=\beta^{\max}=3$, 
\item (\textbf{NGR}): Nesterov's method with the gradient restarting scheme proposed in \cite[Section~3.2]{o2015adaptive} with $q=0$ and $t_k=\frac{1}{L_f}$.
\end{itemize} 
It is evident that the discrete counterparts of our proposed structures perform poorly compared to these algorithms, reinforcing the assertion of Remark~\ref{rem_5} calling for a smarter discretization technique. Observe that \textbf{NGR} provides the best convergence with respect to the other consider methods. In Figure~\ref{fig:obj_disc_bests}, we depict the best behavior of the considered methods (excluding \textbf{NGR}) for this specific example. It is interesting that \textbf{NGR} still outperforms all other methods.

Consider the three methods \textbf{Struct~I}, \textbf{Struct~II}, and \textbf{NSR} in Figure~\ref{fig:obj_disc_all}. The results depicted in Figure~\ref{fig:obj_disc_all} correspond to the standard parameters involved in each algorithm, i.e., the step size $s = 1/L_f$ for the proposed methods in Corollary~\ref{Cor_1}, and the parameter $k_{\min} = 1$ in NSR. As we saw in Figure~\ref{fig:obj_disc_bests}, these parameters can also be tuned depending on the application at hand. In case of NSR, the role of the parameter $k_{\min}$ is to prevent unnecessary restarting instants that may degrade the overall performance. On the other hand, setting $k_{\min}>1$ may potentially cause the algorithm to lose its monotonicity property. Figure~\ref{fig:obj_disc_NSR} shows how changing $k_{\min}$ affects the performance. The best performance is achieved by setting $k_{\min}=19$ and the algorithm becomes non-monotonic for $k_{\min}>19$. With regards to our proposed methods we observe that if one increases the step size $s$, the performance improves, see Figure~\ref{fig:obj_disc_1} for \textbf{Struct I} and Figure~\ref{fig:obj_disc_2} for \textbf{Struct II}. Moreover, it is obvious that the discrete-time couterparts of \textbf{Struct I} and \textbf{Struct II} behave in a very similar fashion that has to do with the lack of a proper discretization that can fully exploit the properties of the corresponding feedback input, see Remark~\ref{rem_5}.

\begin{figure}[t]
    \centering
	\subfigure[\textbf{NSR}]{\label{fig:obj_disc_NSR}\includegraphics[width=53mm]{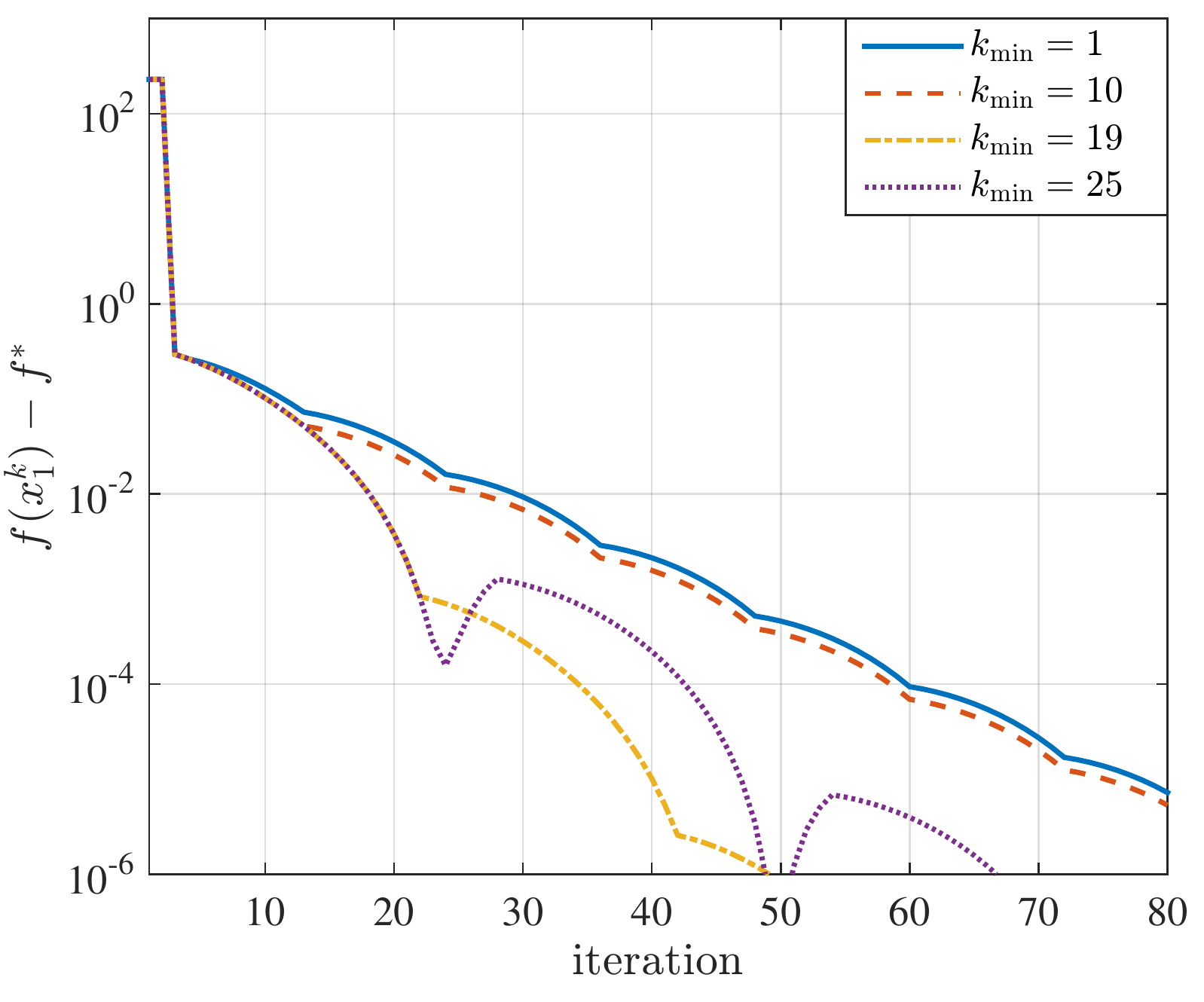}}    
	\subfigure[\textbf{Struct~I}]{\label{fig:obj_disc_1}\includegraphics[width=53mm]{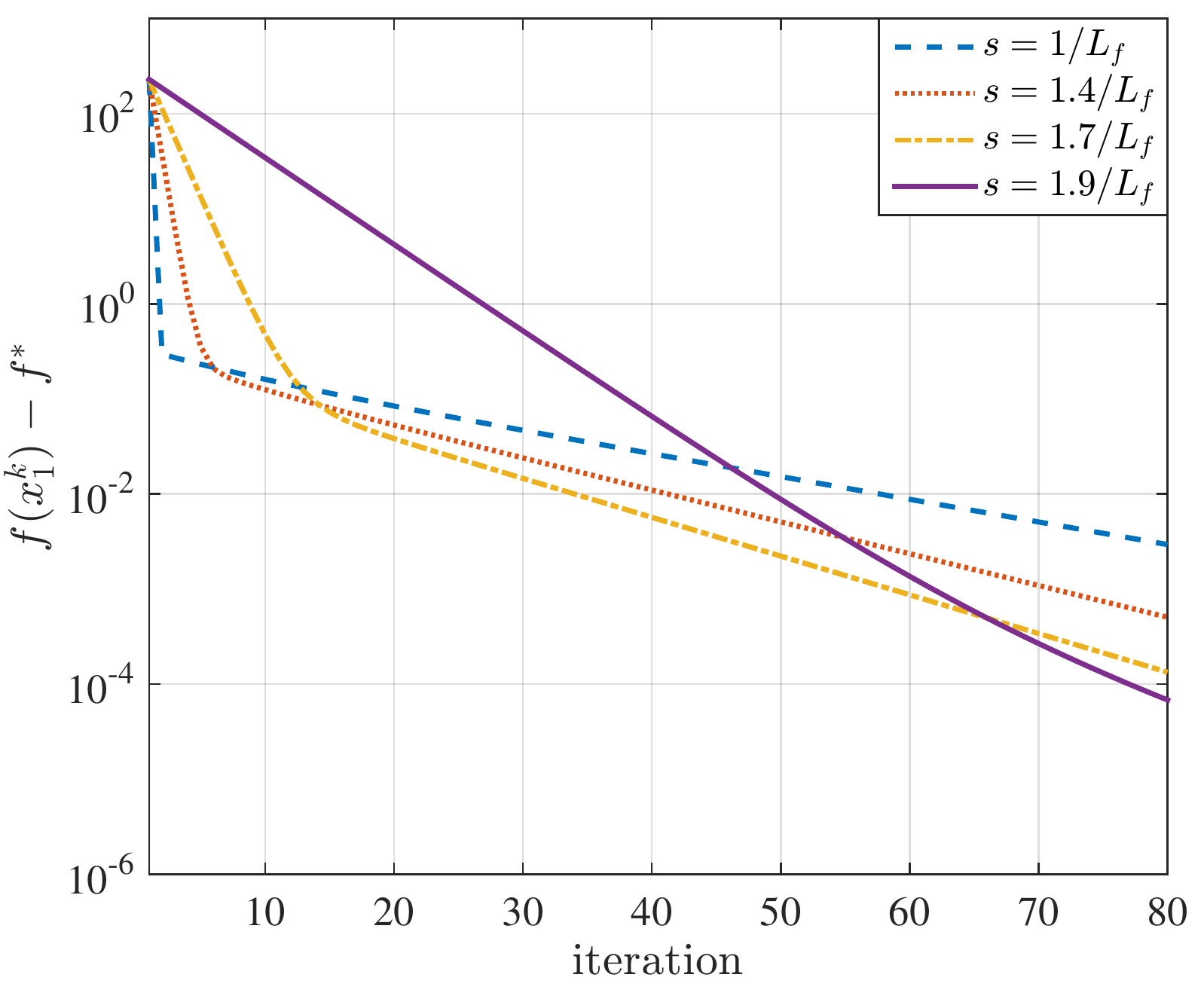}}
	\subfigure[\textbf{Struct~II}]{\label{fig:obj_disc_2}\includegraphics[width=53mm]{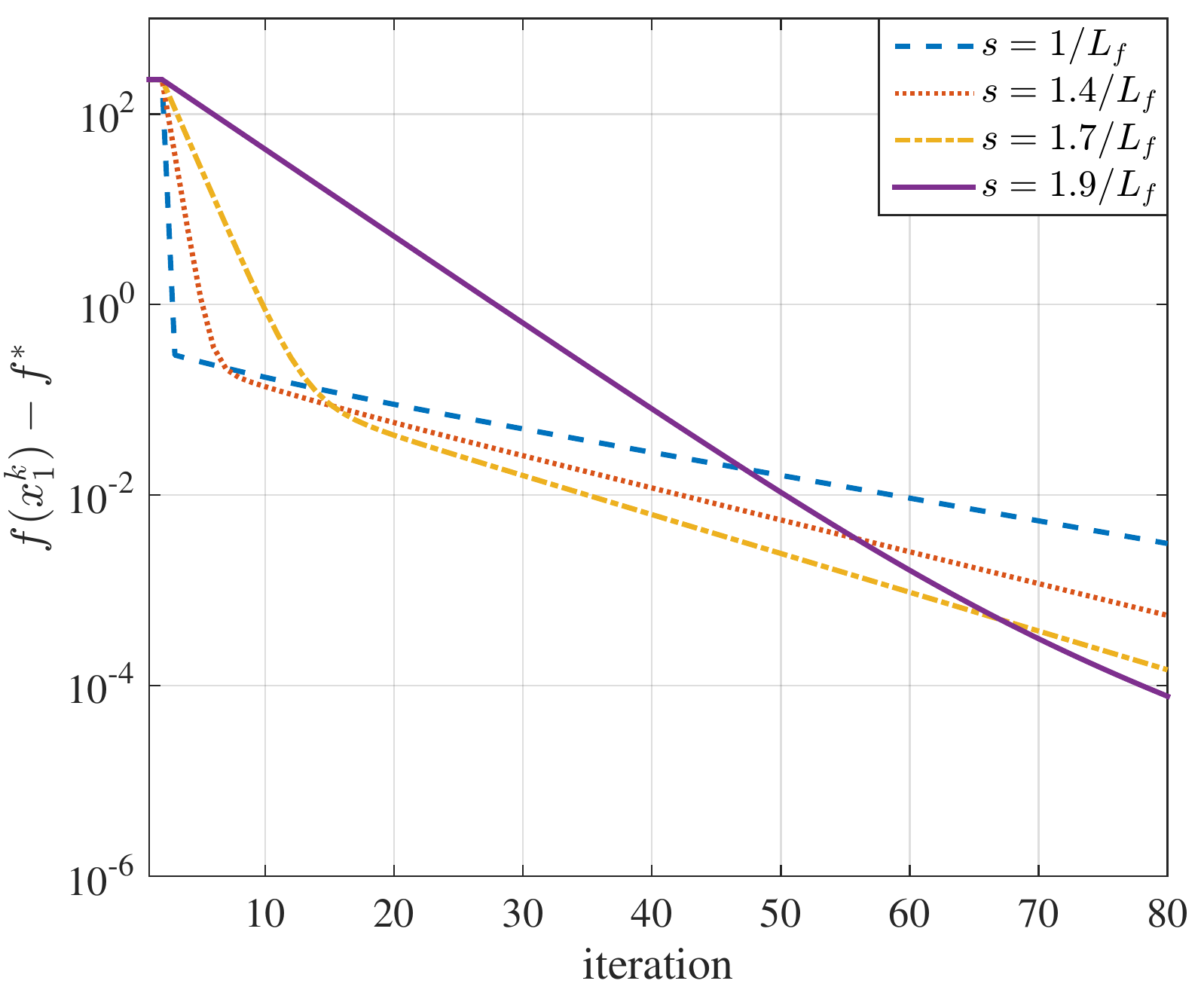}}	
    \caption{Discrete-time dynamics under different tuning parameters.}
    \label{fig:disc_comp_all}
\end{figure}

\section{Conclusions}
\label{sec:conc}
Inspired by a control-oriented viewpoint, we proposed two hybrid dynamical structures to achieve exponential convergence rates for a certain class of unconstrained optimization problems, in a continuous-time setting. The distinctive feature of our methodology is the synthesis of certain inputs in a state-dependent fashion compared to a time-dependent approach followed by most results in the literature. Due to the state-dependency of our proposed methods, the time-discretization of continuous-time hybrid dynamical systems is in fact difficult (and to some extent even more involved than the time-varying dynamics that is commonly used in the literature). In this regard, we have been able to show that one can apply the the forward-Euler method to discretize the continuous-time dynamics and still guarantee exponential rate of convergence. Thus, a more in-depth analysis is due. We expect that because of the state-dependency of our methods a proper venue to search is geometrical types of discretization.

	\bibliographystyle{siam}	
	\bibliography{./mybref}

\end{document}

%% file: style.tex
\makeatletter
\def\@settitle{\begin{center}%
		\baselineskip14\p@\relax
		\normalfont\LARGE\scshape\bfseries
		\@title
	\end{center}%
}
\makeatother

\makeatletter

\def\subsection{\@startsection{subsection}{2}%
	\z@{.5\linespacing\@plus.7\linespacing}{.5\linespacing}%
	{\normalfont\large\bfseries}}

\def\subsubsection{\@startsection{subsubsection}{3}%
	\z@{.5\linespacing\@plus.7\linespacing}{.5\linespacing}%
	{\normalfont\itshape}}

\usepackage[usenames, dvipsnames]{color}
\definecolor{darkblue}{rgb}{0.0, 0.0, 0.45}

\usepackage[colorlinks	= true,
raiselinks	= true,
linkcolor	= darkblue, 
citecolor	= Mahogany,
urlcolor	= ForestGreen,
pdfauthor	= {Arman},
pdftitle	= {},
pdfkeywords	= {},
pdfsubject	= {},
plainpages	= false]{hyperref}

\usepackage{dsfont,amssymb,amsmath,subfigure, graphicx,enumitem,multirow} 
\usepackage{amsfonts,dsfont,mathtools, mathrsfs,amsthm} 
\usepackage[amssymb, thickqspace]{SIunits}
\usepackage{algorithm,algorithmic,fancyhdr}
\allowdisplaybreaks
\date{\today}

%% file: commands.tex
\theoremstyle{theorem}
\newtheorem{Thm}{Theorem}[section]

\newtheorem{Lem}[Thm]{Lemma}
\newtheorem{Cor}[Thm]{Corollary}
\newtheorem{As}[Thm]{Assumption}
\newtheorem{Prob}[Thm]{Problem}

\newtheorem{Def}[Thm]{Definition}

\newtheorem{Rem}[Thm]{Remark}

\theoremstyle{remark}
\newtheorem{Ex}{Example}

\newcommand{\ol}[1]{\overline{#1}}
\newcommand{\ul}[1]{\underline{#1}}

